\documentclass{amsart}
\usepackage[utf8]{inputenc}
\usepackage{amsmath,amsthm}
\usepackage{comment}
\usepackage{amssymb}
\usepackage{stmaryrd}
\usepackage[hidelinks]{hyperref}
\usepackage{xpatch}
\usepackage{color}
\usepackage[left=3cm,right=3cm,top=3cm,bottom=4.5cm]{geometry}
\usepackage{pdfsync}
\usepackage{tikz-cd}
\usepackage{url}
\usepackage{enumitem}
\hypersetup{linktoc=all}
\usepackage[capitalise,noabbrev]{cleveref}

\newcommand{\defterm}[1]{\emph{#1}}

\numberwithin{equation}{section}

\theoremstyle{plain}
	\newtheorem{thm}[equation]{Theorem}
	\newtheorem*{thm*}{Theorem}
  	\newtheorem*{thma*}{Theorem A}
	\newtheorem*{thmb*}{Theorem B}
	\newtheorem*{thmc*}{Theorem C}
	\newtheorem{cor}[equation]{Corollary}
	\newtheorem*{cor*}{Corollary}
	\newtheorem{prop}[equation]{Proposition}
	\newtheorem*{prop*}{Proposition}
	\newtheorem{lem}[equation]{Lemma}
	\newtheorem*{lem*}{Lemma}
	
	\newtheorem*{ex*}{Exercise}
	
	\newtheorem*{claim*}{Claim}
	
	\newtheorem*{question*}{Question}
	
	\newtheorem*{fact*}{Fact}
	
	\newtheorem*{conj*}{Conjecture}

\theoremstyle{definition}
	\newtheorem{Def}[equation]{Definition}
	\newtheorem*{Def*}{Definition}
	\newtheorem{obs}[equation]{Observation}
	\newtheorem*{obs*}{Observation}
	\newtheorem{rmk}[equation]{Remark}
	\newtheorem*{rmk*}{Remark}
	
	\newtheorem{soln*}{Solution}
	
	\newtheorem*{note*}{Note}
	\newtheorem{eg}[equation]{Example}
	\newtheorem*{eg*}{Example}	
	
	\newtheorem*{construction*}{Construction}
	
	\newtheorem*{warning*}{Warning}

\xpatchcmd{\paragraph}{\normalfont}{{\normalfont\bfseries}}{}{}
	
\newcommand{\ints}{\mathbb{Z}}

\newcommand{\nats}{\mathbb{N}}

\newcommand{\id}{\mathrm{id}}

\newcommand{\Hom}{\mathrm{Hom}}

\newcommand{\Ob}{\operatorname{Ob}}
\newcommand{\Mor}{\operatorname{Mor}}

\newcommand{\const}{\mathrm{const}}

\newcommand{\Psh}{\mathsf{Psh}}
\newcommand{\Fun}{\mathrm{Fun}}

\newcommand{\Set}{\mathsf{Set}}

\newcommand{\Cat}{\mathsf{Cat}}

\newcommand{\calC}{\mathcal{C}}

\newcommand{\calE}{\mathcal{E}}

\newcommand{\calK}{\mathcal{K}}

\newcommand{\calO}{\mathcal{O}}

\newcommand{\calS}{\mathcal{S}}

\newcommand{\Gaunt}{\mathsf{Gaunt}}

\newcommand{\sCat}{\mathsf{sCat}}

\newcommand{\globe}{\mathbb{G}}

\newcommand{\plain}{\mathrm{plain}}
\newcommand{\semi}{\mathrm{semi}}

\newcommand{\TPar}{\mathcal{S}}

\newcommand{\inc}{\mathrm{f}}
\newcommand{\nerve}{\operatorname{N}}
\newcommand{\lax}{\mathrm{lax}}
\newcommand{\strong}{\mathrm{strong}}
\newcommand{\cocts}{\mathrm{cocts}}
\newcommand{\funny}{\mathrm{funny}}

\newcommand{\Lev}{\mathsf{Lev}}
\newcommand{\Levw}{\mathsf{Lev}_\infty}
\newcommand{\natsw}{\mathbb{N}_\infty}
\newcommand{\Alg}{\operatorname{Alg}}
\newcommand{\coind}{\mathrm{coind}}

\DeclareFontFamily{U}{min}{}
\DeclareFontShape{U}{min}{m}{n}{<-> udmj30}{}
\newcommand\yo{\!\text{\usefont{U}{min}{m}{n}\symbol{'207}}\!}

\title{The Gray tensor product of $(\infty,n)$-categories}
\author{Timothy Campion}
\date{October 2023}

\begin{document}

\begin{abstract}
    In this note, we leverage the author's pasting theorem for $(\infty,n)$-categories to construct new models of $(\infty,n)$-categories for all $n \leq \infty$, as presheaves on certain categories of computads. Among these new models are some which facilitate a construction of the Gray tensor product of $(\infty,n)$-categories Day convolution and reflection, which we carry out here. After constructing this Gray tensor product of $(\infty,n)$-categories, we characterize it via several model-independent universal properties.
\end{abstract}

\maketitle

\setcounter{tocdepth}{1}
\tableofcontents


\section{Introduction}
In this note, we apply the pasting theorem of \cite{campion-paste} to construct new models of $(\infty,n)$-categories for all $n \leq \infty$, as presheaves of spaces on certain categories of computads, and to describe the resulting localizations quite explicitly. This work should be compared to various presheaf models of $(\infty,n)$-categories found in \cite{bsp}, or in \cite{henry-regular}. As an application, the good control over these localizations allows us to construct the Gray tensor product of $(\infty,n)$-categories via Day convolution and reflection quite straightforwardly, as a direct application of results from the strict $n$-categorical literature (\cite{ara-lucas}). We then characterize the resulting Gray tensor product via several model-independent universal properties (which in particular do not refer to the site we used to construct the Gray tensor product).

\subsection*{Context} The Gray tensor product of strict 2-categories was introduced in \cite{Gray} in order to study lax natural transformations. It is a monoidal biclosed structure $\otimes$ on the 1-category $\sCat_2$ of strict 2-categories, such that the left (resp. right) internal hom $\llbracket A, B\rrbracket$ has strict 2-functors $A \to B$ for objects, lax (resp. oplax) natural transformations for 1-morphisms, and modifications for 2-morphisms. The unit of the monoidal structure is the terminal category $\square^0$, and $\Ob(A \otimes B) = \Ob A \times \Ob B$. The first few Gray tensor powers of the arrow category $[1]$ are pictured in \cref{cubes-0}. Gray's construction was extended to the category $\sCat_\omega$ of strict $\omega$-categories using pasting schemes \cite{Crans:thesis}, cubical $\omega$-categories \cite{al-Agl;Brown;Steiner:Multiple} and augmented directed complexes \cite{steiner}. It was shown in \cite{ara-lucas} that the ``folk" model structure on strict $\omega$-categories \cite{lmw} is monoidal for the Gray tensor product. In this setting, the power $[1]^{\otimes n}$ is a lax $n$-dimensional cube; these are the \defterm{Gray cubes} $\square^n =  [1]^{\otimes n}$, and the collection thereof is the \defterm{Gray cube category} $\square = \{\square^n \mid n \in \nats\}$, forming a monoidal subcategory of the 1-category $\sCat_\omega$ of strict $\omega$-categories.

\begin{figure}
\[
\begin{tikzpicture}[scale = 1.5, baseline = -2]
	\filldraw
	(0,0) circle [radius = 1pt]
	(1,0) circle [radius = 1pt];
	
	\draw[->] (0.1,0) -- (0.9,0);

\end{tikzpicture} \hspace{50pt}
\begin{tikzpicture}[scale = 1.5, baseline = 18.5]
	\filldraw
	(0,0) circle [radius = 1pt]
	(1,0) circle [radius = 1pt]
	(0,1) circle [radius = 1pt]
	(1,1) circle [radius = 1pt];
	
	\draw[->] (0.1,0) -- (0.9,0);
	\draw[->] (0.1,1) -- (0.9,1);
	\draw[<-] (0,0.1) -- (0,0.9);
	\draw[<-] (1,0.1) -- (1,0.9);
	
	\draw[->, double] (0.7,0.7) -- (0.3,0.3);
\end{tikzpicture} \hspace{50pt}
\begin{tikzpicture}[baseline = -2,scale = 1.5]	
	\filldraw
	(150:1) circle [radius = 1pt]
	(90:1) circle [radius = 1pt]
	(30:1) circle [radius = 1pt]
	(-30:1) circle [radius = 1pt]
	(-90:1) circle [radius = 1pt]
	(-150:1) circle [radius = 1pt]
	(0,0) circle [radius = 1pt];
	
	\draw[->] (150:1) + (30:0.1) --+ (30:0.9);
	\draw[->] (90:1) + (-30:0.1) --+ (-30:0.9);
	\draw[->] (30:1) + (-90:0.1) --+ (-90:0.9);
	\draw[->] (150:1) + (-90:0.1) --+ (-90:0.9);
	\draw[->] (-150:1) + (-30:0.1) --+ (-30:0.9);
	\draw[->] (-90:1) + (30:0.1) --+ (30:0.9);

	\draw[->] (150:1) + (-30:0.1) --+ (-30:0.9);
	\draw[->] (0:0) + (-90:0.1) --+ (-90:0.9);
	\draw[->] (0:0) + (30:0.1) --+ (30:0.9);
	
	\draw[<-,double] (-150:0.5) + (-0.15,-0.15) --+ (0.15,0.15);
	\draw[<-,double] (90:0.5) + (-0.15,-0.15) --+ (0.15,0.15);
	\draw[<-,double] (-30:0.5) + (-0.15,-0.15) --+ (0.15,0.15);
\end{tikzpicture}
\quad \Rrightarrow \quad
\begin{tikzpicture}[baseline = -2,scale = 1.5]
	\filldraw
	(150:1) circle [radius = 1pt]
	(90:1) circle [radius = 1pt]
	(30:1) circle [radius = 1pt]
	(-30:1) circle [radius = 1pt]
	(-90:1) circle [radius = 1pt]
	(-150:1) circle [radius = 1pt]
	(0,0) circle [radius = 1pt];
	
	\draw[->] (150:1) + (30:0.1) --+ (30:0.9);
	\draw[->] (90:1) + (-30:0.1) --+ (-30:0.9);
	\draw[->] (30:1) + (-90:0.1) --+ (-90:0.9);
	\draw[->] (150:1) + (-90:0.1) --+ (-90:0.9);
	\draw[->] (-150:1) + (-30:0.1) --+ (-30:0.9);
	\draw[->] (-90:1) + (30:0.1) --+ (30:0.9);

	\draw[->] (0:0) + (-30:0.1) --+ (-30:0.9);
	\draw[->] (90:1) + (-90:0.1) --+ (-90:0.9);
	\draw[->] (-150:1) + (30:0.1) --+ (30:0.9);
	
	\draw[<-,double] (150:0.5) + (-0.15,-0.15) --+ (0.15,0.15);
	\draw[<-,double] (-90:0.5) + (-0.15,-0.15) --+ (0.15,0.15);
	\draw[<-,double] (30:0.5) + (-0.15,-0.15) --+ (0.15,0.15);
\end{tikzpicture}
\]
\caption{The first few Gray cubes: $\square^1 = [1]$, $\square^2 = [1] \otimes [1]$, and $\square^3 = [1] \otimes [1] \otimes [1]$}\label{cubes-0}
\end{figure}
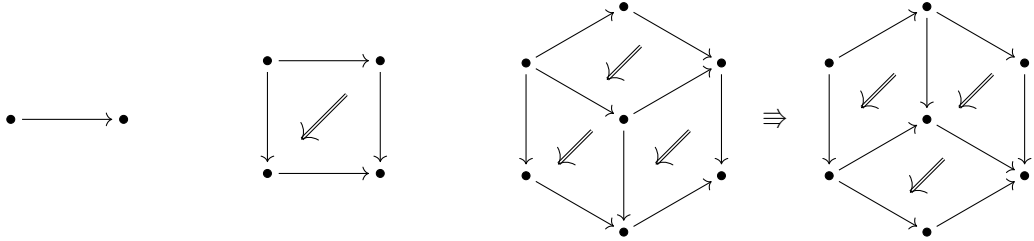

Several authors have constructed related monoidal structures on the $\infty$-category $\Cat_n$ of (weak) $(\infty,n)$-categories for general $n \in \nats \cup \{\omega\}$. We refer to the introduction of \cite{campion-maehara} for a more comprehensive discussion, but we mention here that the complicial model \cite{Verity:I} and comical models (\cite{Campion;Kapulkin;Maehara}, \cite{Doherty;Kapulkin;Maehara}) have ``Gray tensor products," which agree (at least at the level of bifunctors) under the comparison of \cite{Doherty;Kapulkin;Maehara}. Thanks to the work of \cite{Loubaton} (\cite{Gagna;Harpaz;Lanari:Equivalence} in dimension 2) and \cite{Doherty;Kapulkin;Maehara} these models have moreover been compared to models which are universal in the sense of \cite{bsp}. The relevant combinatorics were also studied in detail by \cite{Johnson-Freyd;Scheimbauer} in iterated complete Segal spaces, but that work has not been compared to other models of the Gray tensor product. There has also been much interesting work specific to dimension 2 -- see e.g. \cite{Gaitsgory;Rozenblyum} \cite{Maehara:Gray} \cite{Gagna;Harpaz;Lanari:Gray} \cite{campion-maehara}. The present work is similar in spirit to \cite{campion-maehara}.

\subsection*{Outlook} However, much remains to be understood about the Gray tensor product of $(\infty,n)$-categories. Currently, for $n \geq 3$ the only constructions in the literature are in the complicial and comical models. These can be transferred to other models using the comparison of \cite{Loubaton}, but it would be preferable to have a construction living more natively to these other models. Moreover, the justification for even referring to this ``Gray tensor product" as ``the Gray tensor product" is a bit tenuous. For instance, there is no known comparison between this ``Gray tensor product" and the usual Gray tensor product of strict $n$-categories. In dimension 2, these shortcomings were addressed in \cite{Maehara:Gray} by constructing a ``Gray tensor product" by extending up from the usual Gray tensor product on a class of strict 2-categories. In \cite{campion-maehara} it was shown that this ``Gray tensor product" descends from the point-set-level to $\Cat_2$ and enjoys a couple of related universal properties. For one, it is the unique monoidal biclosed structure extended from the usual Gray tensor product of Gray cubes. For another, a cocontinuous, strong monoidal functor out of $\Cat_2$ is the same as a cocontinuous functor out of $\Cat_2$ plus a strong monoidal structure on its restriction to the Gray cubes. These universal properties make no mention of the particular construction of the Gray tensor product in \cite{Maehara:Gray}, but their proof is facilitated by the knowledge, provided by \cite{Maehara:Gray}, that this ``Gray tensor product" can be constructed as the Day reflection of the Day convolution of the usual Gray tensor product on the Gray cubes.

\subsection*{Present work} In this note, we carry out a similar program to \cite{campion-maehara} in arbitrary dimension, constructing the Gray tensor product and then giving a characterization which abstracts from the construction.
We show:

\begin{thma*}[See \cref{cor:gray-reflect} and \cref{cor:gray-unique1}]
Let $n \in \nats \cup \{\omega\}$. There is a unique monoidal biclosed structure on $\Cat_n$ such that the canonical functor $\square \subset \sCat_\omega \to \Cat_\omega \to \Cat_n$ is strong monoidal.
\end{thma*}

We call this monoidal structure the \defterm{Gray tensor product} of $(\infty,n)$-categories. Theorem A tells us that the Gray tensor product of $(\infty,n)$-categories exists, and it is uniquely determined by its restriction to the Gray cubes. Several remarks are in order.

\begin{rmk*}
    The characterization of the Gray tensor product given Theorem A may seem circular. Of course, it is not in fact circular, since the strict Gray tensor product may be defined independently of the weak Gray tensor product. Better still, the strict Gray tensor product itself is typically defined in a similar two-step fashion (cf. \cite{Crans:thesis} \cite{al-Agl;Brown;Steiner:Multiple} \cite{steiner} \cite{ara-maltsiniotis}), by first giving a definition for certain combinatorially well-behaved $n$-categories such as cubes, parity complexes, or augmented directed complexes, and then extending to all strict $n$-categories by colimits. Theorem A may be applied without carrying out the second step in the strict world -- it suffices to understand the Gray tensor product of cubes. Of course, the Gray tensor product of cubes already encodes a great deal of intricate combinatorics, and it would be desirable to find some other characterization with even less in the way of combinatorial prerequisites. We give some progress in this direction in \cref{thm:gray-plain}, where we show that in the statement of Theorem A, at least when $n = \omega$, we may replace $\square$ with the non-full subcategory $\square_\semi$ of Gray cubes and subcomplex inclusions (equivalent to the category of plain semicubical sets -- see \cref{def:semi}).
\end{rmk*}

\begin{rmk*}
    Note that we have in fact supplied several ``Gray tensor products" -- one on $\Cat_n$ for each $n$. In the body of the paper, we supply many more (including e.g. a Gray tensor product of $(m,n)$-categories, where $m < \omega$, and a Gray tensor product of $(\infty,\infty)$-categories with coinductive equivalences). These all originate from the universal case (\cref{thm:day-refl}) where $n = \omega$, and in the slightly more general setting of \defterm{flagged} $(\infty,\infty)$-categories (cf. \cite{ayala-francis}), the $\infty$-category of which we denote $\Cat_\omega^\inc$ (see \cref{rmk:to-infty}). We then show that various localizations of $\Cat_\omega^\inc$ such as $\Cat_n$ are exponential ideals (\cref{cor:gray-reflect}), and Theorem A results.
\end{rmk*}

\begin{rmk*}
    The uniqueness part of Theorem A is straightforward to deduce from the fact that $\square$ is dense in $\Cat_\omega^\inc$ \cite{campion-dense}. The main import is the existence.
\end{rmk*}

\subsection*{Method} The construction of the primordial Gray tensor product on $\Cat_\omega^\inc$ is by Day convolution (\cref{thm:day-conv}) and reflection (\cref{thm:day-refl}). That is, we follow the following steps:

\begin{enumerate}
    \item We identify a dense subcategory $\calS \subset \Cat_\omega^\inc$ on which we already know how to define the Gray tensor product.
    \item We take the Day convolution on $\Psh(\calS)$.
    \item We show that the image $\Cat_\omega^\inc \subset \Psh(\calS)$ is an exponential ideal so that the tensor product descends. 
\end{enumerate}

\begin{rmk*}
    In Step (1), we call such an $\calS$ a \defterm{monoidally suitable site} (\cref{def:mon-suit}). For example, we may take $\calS$ to comprise the \defterm{strongly loop-free Steiner complexes} (see \cref{eg:mon-suit}). By definition, a monoidally suitable site is a full subcategory of the \defterm{torsion-free complexes} of \cite{forest}, which in turn is a full subcategory of $\sCat_\omega$, such that $\Theta \subseteq \calS$ and $\calS$ is closed under the strict Gray tensor product.\footnote{Note that the inclusion $\sCat_\omega \to \Cat_\omega^\inc$ is fully faithful, unlike the functor $\sCat_\omega \to \Cat_\omega^\inc \to \Cat_\omega$, so that the inclusion $\calS \to \sCat_\omega^\inc$ is also fully faithful. Of course, the torsion-free complexes are all gaunt, and the inclusion $\Gaunt_\omega \to \Cat_\omega$ is also fully faithful, so that $\calS \to \Cat_\omega$ is in fact fully faithful, but we do not use that here.} The requirement that $\Theta \subseteq \calS$ ensures that $\calS$ is dense in $\Cat_\omega^\inc$. Here $\Theta$ is Joyal's category $\Theta$.
\end{rmk*}

\begin{rmk*}
    For Step (3), it is important not merely to know that $\calS$ is dense in $\Cat_\omega^\inc$, but to have a good enough understanding of the localization $\Psh(\calS) \to \Cat_\omega^\inc$ to verify the exponential ideal property. The reason we work with torsion-free complexes is in order to invoke the main result of \cite{campion-paste}. This allows us to describe the localization explicitly as in Theorem B below, and deduce the exponential ideal property directly from known results in the strict $\omega$-cateogrical literature \cite{ara-lucas}.
\end{rmk*}

That is, we have in the course of the construction given a new family of models for $(\infty,n)$-categories. This in fact works without any monoidal assumptions: we define a \defterm{suitable site} to be any full subcategory $\calS$ of the category of torsion-free complexes which contains $\Theta$. We have:

\begin{thmb*}[see \cref{def:j-square} and \cref{cor:locn}]
    Let $\calS$ be a suitable site. Then the localization $\Psh(\calS) \to \Cat_\omega^\inc$ is generated by the maps
    \[\yo A \cup_{\yo \partial \globe_n} \yo \globe_n \to \yo B.\]
    for each pushout square in $\sCat_\omega$ of the form
    \begin{equation*}
        \begin{tikzcd}
            \partial \globe_n \ar[r,hook] \ar[d,tail] & \globe_n \ar[d, tail] \\
            A \ar[r,hook] & B
        \end{tikzcd}
    \end{equation*}
    which lies in $\calS$.
\end{thmb*}

The density of $\calS$ follows by definition; the key input to this description of the localization is the main theorem of \cite{campion-paste}, which allows us to describe the localization:
it is the universal localization forcing ``cell attachment pushouts" in $\calS$ to remain pushouts in $\Cat_\omega^\inc$. Before \cite{campion-paste} it was not known that these pushouts were in fact preserved by the inclusion $\sCat_\omega \to \Cat_\omega^\inc$ from strict $\omega$-categories to flagged $(\infty,\infty)$-categories. Knowing the generators of this localization is what allows us to verify that $\Cat_\omega^\inc$ is an exponential ideal in $\Psh(\calS)$ (\cref{thm:day-refl}) by directly quoting \cite{ara-lucas}'s result that the ``folk" model structure on $\sCat_\omega$ of \cite{lmw} is monoidal for the usual Gray tensor product. In fact, we only need a fragment of the result in \cite{ara-lucas}, namely the compatibility of the Gray tensor product with pushouts along folk cofibrations. 

\begin{rmk*}
    The models constructed in Theorem B are similar in spirit to the discussion of presheaf models in \cite{bsp} and the presheaf models on other sites considered in \cite{henry-regular}. When considering such models, there is generally a trade-off to be made: working with a larger site may be convenient, but generally comes at the cost of a less explicit description of the relevant localization. For instance, presheaves on Joyal's category $\Theta_n$ can be localized explicitly at the spine inclusions to get flagged $(\infty,n)$-categories, but $\Theta_n$ is among the smallest sites one might consider. Larger sites, such as the site $\Upsilon_n$ considered in \cite{bsp}, lead to less explicit descriptions of the localization. The homotopy pushouts considered in the present paper will allow to explicitly describe these localizations when working with larger sites than has previously been feasible.
\end{rmk*}

\begin{rmk*}
    Note that in the localization described in Theorem B, there is no need to explicitly include ``Segal" conditions having to do with spines. These follow from the stated cell-attachment conditions.
\end{rmk*}

\begin{rmk*}
    Models of other categories such as $\Cat_n$ follow by further localization (see \cref{cor:locn'}). 
\end{rmk*}

Finally, we give a universal property of the Gray tensor product:
\begin{thmc*}[See \cref{cor:gray-reflect} and \cref{cor:gray-unique1}]
    Let $n \in \nats \cup \{\omega\}$, and regard $\Cat_n$ as monoidal under the Gray tensor product. Let $\calE$ be monoidal biclosed and cocomplete. Then the restriction map from the space of strong monoidal cocontinuous functors $\Cat_n \to \calE$, to the space of cocontinuous functors $\Cat_n \to \calE$ with a strong monoidal restriction to $\square$, is an equivalence.
\end{thmc*}

As with Theorem A, Theorem C applies to other related categories as well (see \cref{cor:gray-unique1}). In fact, given the existence of the Gray tensor product and the density of $\square \subset \Cat_\omega^\inc$, both Theorem A and Theorem C follow easily from the standard properties of the Day convolution. 
Both Theorem A and Theorem C should be compared to the results of \cite{campion-maehara} in dimension 2. See \cref{rmk:shortcomings} for a discussion of the subtle differences.

\subsection*{Future work} It remains to compare various ``Gray tensor products" appearing in the literature. Just as in the case $n = 2$ in \cite{campion-maehara}, we now have ``Gray tensor product" with a universal property, but just as in that case, it may be difficult to verify that other constructions in the literature actually enjoy this universal property. We also leave the development of basic properties of the Gray tensor product, comparison to the cartesian product, etc. to future work.

\subsection*{Outline} We begin in \cref{sec:loc} by reviewing the $\infty$-category $\Cat_\omega^\inc$ of flagged $(\infty,\infty)$-categories and various localizations thereof, starting from the $\Theta$-space model. We introduce (see esp. \cref{def:rezkloc} and \cref{def:coindloc}) some notation for discussing a particularly interesting family of such localizations which includes e.g. $\Cat_n$ for $n \leq \omega$ and $\Cat_\omega^{\coind}$ with coinductive weak equivalences. Next, in \cref{sec:sites} we show how to construct all of these $\infty$-categories as localizations of presheaves on a \defterm{suitable site} (\cref{def:suitable}). In \cref{sec:gray}, we apply our knowledge of these new models to the case of a \defterm{monoidally suitable site} (\cref{def:mon-suit}) to construct (via Day convolution and reflection) the Gray tensor product on $\Cat_\omega^\inc$ (\cref{thm:day-refl}) and on the most important of the localizations thereof from the previous section (\cref{cor:gray-reflect}). Finally, in \cref{sec:unique} we prove that the resulting Gray tensor product enjoys several (model-independent) universal properties as described above.

\subsection{Notation}
We write $\Psh(\calC)$ for the $\infty$-category of presheaves of spaces on an $\infty$-category $\calC$. We write $A \ast_S G$ for a pushout of $A \leftarrow S \to G$ taken in the 1-category $\sCat_\omega$ of strict $\omega$-categories. If $S$ is a class of morphisms in a cocomplete $\infty$-category $\calC$, then $L_S \calC$ denotes the cocontinuous localization at this class of morphisms, and $L_S : \calC \to L_S \calC$ the localization functor. We denote by $S^n$ the $n$-sphere, and in particular $S^{-1} = \emptyset$. If $n \in \nats$, we denote $[n] = \{0,1,\dots, n\}$ and $[\omega] = \nats$ for the corresponding ordinal. We write $\Lev = \ints_{\geq -2} = \{n \in \ints \mid n \geq -2\}$, $\Levw = \Lev \cup \{\omega\}$ and $\natsw = \nats \cup \{\omega\}$. We recall that in ordinal arithmetic we have $1 + \omega = \omega$. We denote $\Pr^L$ the $\infty$-category of presentable $\infty$-categories and left adjoint functors. We write $\yo$ for the Yoneda embedding.

\subsection{Acknowledgements}
Thanks to Alexander Campbell, Yuki Maehara, and Emily Riehl for helpful discussions. I'm grateful for the support of the ARO under MURI Grant W911NF-20-1-0082.

\section{The $\infty$-category of flagged $(\infty,\infty)$-categories and localizations thereof}\label{sec:loc}
In this section, we discuss various $(\infty,1)$-categories of $(\infty,n)$-categories. Beginning in \cref{subsec:flagged}, we ground our discussion in the $\Theta$-space model of \cite{rezk}, but the main points are all model-independent. We begin by discussing, for each $n \in \natsw$, the $\infty$-category $\Cat_n^\inc$ of \defterm{flagged $(\infty,n)$-categories} of \cite{ayala-francis}. We discuss their relations and how to model them as $\Theta_m$-spaces for $m \geq n$. The case $n = \omega$ does not appear in \cite{ayala-francis}, but presents no real additional difficulty. Then in \cref{subsec:loc}, we localize the categories $\Cat_n^\inc$ to obtain various $\infty$-categories of $(\infty,n)$-categories with different truncation conditions (\cref{def:truncloc}), different Rezk completeness conditions (\cref{def:rezkloc}), and with the option for coinductive equivalences (\cref{def:coindloc}). We introduce notation for such localizations, the most general of which is denoted $\Cat_{\underline{d}}^{\inc < R}$ or $\Cat_{\underline{d}}^{\inc <^\coind C}$. We point out many examples.

\subsection{Flagged $(\infty,n)$-categories as presheaves on $\Theta$}\label{subsec:flagged}
We begin our discussion of higher categories with a discussion of \defterm{flagged $(\infty,n)$-categories} (cf \cite{ayala-francis}). This includes the case $n = \omega$ (cf. \cref{rmk:to-infty}), which gives the most general $\infty$-category of higher $\infty$-categories which we will consider -- the $\infty$-category $\Cat_\omega^\inc$ of \defterm{flagged $(\infty,\infty)$-categories}. All of the other $\infty$-categories of higher $\infty$-categories we shall consider are localizations of this one. We include a discussion (\cref{lem:same-loc}) of the Segal conditions used to model these categories as $\Theta_n$-spaces.

\begin{Def}
For $n \in \natsw$, let $\sCat_\omega$ denote the strict 1-category of strict $\omega$-categories. That is, for $n \in \nats$, we inductively define $\sCat_0 = \Set$, and $\sCat_{1+n}$ is the strict 1-category of small categories enriched in $\sCat_n$. There are inclusion functors $\sCat_n \to \sCat_{1+n}$, and right adjoint forgetful functors $\sCat_{1+n} \to \sCat_n$. The limit of these forgetful functors for $n \in \nats$ is $\sCat_\omega$.
\end{Def}

\begin{Def}\label{def:susp-strict}
For $n \in \natsw$, let $\Sigma : \sCat_n \to \sCat_{1+n}$ denote the suspension functor. Here $\Sigma C$ is the category with two objects, $0$ and $1$, and $\Hom(0,0) = \Hom(1,1) = [0]$, $\Hom(0,1) = C$, and $\Hom(1,0) = \emptyset$. Note that there are natural lifts to adjunctions between the coslice categories 
\[\Sigma_{/} : \sCat_n \rightleftarrows (\sCat_{1+n})_{\partial [1] /} : \Hom\] and between the slice/coslice categories
\[\Sigma_{//} : \sCat_n \rightleftarrows (\sCat_{1+n})_{\partial [1] / / [1]} : \Hom_/\]
where $\Hom(\partial [1] \xrightarrow {x, y} C) = \Hom_C(x,y)$ and $\Hom_/(\partial [1] \xrightarrow {x, y} C \to [1]) = \Hom_C(x,y)$. 
Here the ordinal $[1]$ is regarded as an ordered set and hence a category, and $\partial [1] = [0] \amalg [0]$ is the discrete subcategory with the same objects as $[1]$.
\end{Def}

\begin{Def}\label{def:wedge}
Let $C$ be an $\infty$-category, $c \in C$ an object (the terminal object if unspecified), and $C_{c \amalg c /} = C_{c/} \times_C C_{c/}$ the $\infty$-category of $c$-bipointed objects in $C$. The \defterm{wedge sum} of $c \xrightarrow{d_-} d \xleftarrow{d_+} c, c \xrightarrow{e_-} e \xleftarrow{e_+} c \in C_{c\amalg c / }$ (if it exists) is the object $c \xrightarrow{i_- d_-} d \cup_c e \xleftarrow{i_+ d_+} c$, where $d \cup_c e$ is the pushout of $d \xleftarrow{d_+} c \xrightarrow{d_-} e$, and $d \xrightarrow{i_-} d \cup_c e \xleftarrow{i_+} e$ are the pushout inclusions. We denote this wedge sum $c \vee d$ when there is no confusion over the bipointing, and $c \vee^\calC d$ when we wish to emphasize the category in which the pushout is taken.
\end{Def}

\begin{Def}
Let $\Theta$ denote Joyal's category $\Theta$ (see e.g. \cite{berger}). There are many possible equivalent definitions of $\Theta$; here is one. First, $\Theta_{\ast \ast}$ is the smallest full subcategory of $(\sCat_\omega)_{\partial [1] / }$ which contains the terminal category $[0]$ and is closed under $\Sigma_/$ and wedge sum. Then, $\Theta$ is the full subcategory of $\sCat_\omega$ spanned by the image of the objects of $\Theta_{\ast,\ast}$ under the forgetful functor $(\sCat_\omega)_{\partial [1] / } \to \sCat_\omega$. For $n \in \natsw$, we write $\Theta_n = \Theta \cap \sCat_n$ (so that $\Theta = \Theta_\omega = \cup_{n \in \nats} \Theta_n$).
\end{Def}

\begin{rmk}\label{rmk:colim}
Note that for $n \in \nats$ we have inclusion / truncation adjunctions $\Theta_{n+1} \rightleftarrows \Theta_n$. The colimit of the inclusion functors is $\Theta = \Theta_\omega$ itself. 
By precomposition, there are induced inclusion / forgetful adjunctions $\Psh(\Theta_n) \rightleftarrows \Psh(\Theta_{n+1})$.\footnote{There is also a further left adjoint (truncation) and a further right adjoint giving an adjoint quadruple, but we ignore these.}
The limit along the forgetful functors (ranging over finite $n$) is $\Psh(\Theta)$. Equivalently, $\Psh(\Theta)$ is the colimit in $\Pr^L$ of the inclusion functors $\Psh(\Theta_n) \to \Psh(\Theta_{n+1})$. 
\end{rmk}

We are interested in modeling $(\infty,n)$-categories as presheaves (of spaces) on $\Theta_n$ for $n \leq \omega$. To this end, we begin by lifting the suspension functor to $\Psh(\Theta_n)$.

\begin{Def}\label{def:susp}
Let $n \in \natsw$. Let $\Sigma_/ : \Theta_n \to (\Theta_{1+n})_{/[1]}$ be the suspension functor. Note that $\Sigma_{/}$ admits two canonical lifts $\Sigma^0_/,\Sigma^1_/ : \Theta_n \rightrightarrows (\Theta_{1+n})_{\ast //[1]}$ where $\ast \in \Theta_0$ is the terminal object. Let $\bar \Sigma_/ = \yo_{(\Theta_{1+n})_{/[1]}} \Sigma_/ : \Theta_n \to \Psh((\Theta_{1+n})_{/[1]})$. As the Yoneda embedding preserves the terminal object and presheaves commute with slicing, we have also two lifts $\bar \Sigma^0_/ , \bar \Sigma^1_/ : \Theta_n \rightrightarrows \Psh(\Theta_{1+n})_{\ast//[1]}$. Taking the pullback, we obtain a functor $\bar \Sigma_{//} : \Theta_n \to \Psh(\Theta_{1+n})_{\ast//[1]} \times_{\Psh(\Theta_{1+n})_{/[1]})} \Psh(\Theta_{1+n})_{\ast//[1]} \simeq \Psh(\Theta_{1+n})_{\partial [1]//[1]}$ to bipointed $\Theta_{1+n}$-spaces over $[1]$. Moreover, this functor is fully faithful. There is a unique colimit-preserving extension $\Sigma_{//}$, again fully faithful, and a right adjoint $\Hom_/$:
\[\tilde \Sigma_{//} : \Psh(\Theta_n) \rightleftarrows \Psh(\Theta_{1+n})_{\partial [1] //[1]} : \Hom_/\]
We let $\tilde \Sigma : \Psh(\Theta_n) \to \Psh(\Theta_{1+n})$ denote the composite of $\tilde \Sigma_{//}$ with the forgetful functor.
\end{Def}

\begin{Def}\label{def:spine}
Let $n \in \natsw$. The set of \defterm{basic wedge inclusions} \cite[Notation 12.1]{bsp} is the following set of morphisms in $\Psh(\Theta_n)$: \begin{align*}
    \tilde \Sigma^k(\yo(\theta) \vee^{\Psh(\Theta_n)} \yo(\zeta)) \to \yo(\Sigma^k(\theta \vee^{\sCat_n^\inc} \zeta)) \qquad \text{for $\theta, \zeta \in \Theta_{n-k}$}
\end{align*}

The set of \defterm{basic spine inclusions} \cite[Section 5]{rezk} is the following set of morphisms in $\Psh(\Theta_n)$:
\begin{align*}
    \tilde \Sigma^k(\yo(\Sigma \theta_1) \vee^{\Psh(\Theta_n)} \cdots \vee^{\Psh(\Theta_n)} \yo(\Sigma \theta_r)) \to \yo(\Sigma^k(\Sigma \theta_1 \vee^{\sCat_n} \vee \cdots \vee^{\sCat_n} \Sigma \theta_r)) 
\end{align*}
for $\theta = \Sigma \theta_1 \vee \cdots \vee \Sigma \theta_r \in \Theta_{n-k}$.
\end{Def}

\begin{rmk}
    From the adjunctions $\Sigma_{//} \dashv \Hom_/$ and $\tilde \Sigma_{//} \dashv \Hom_/$ in \cref{def:susp-strict} and \cref{def:susp} we see that $\Sigma$ and $\tilde \Sigma$ preserve contractible colimits, and in particular pushouts. Since these functors do not preserve terminal objects, they do not preserve the wedge sums appearing in \cref{def:spine}, but they do carry them to pushouts in $\Psh(\Theta_n)$ and $\sCat_n$ respectively.
\end{rmk}

\begin{lem}[\cite{bsp}]\label{lem:same-loc}
Let $n \in \natsw$. The localization of $\Psh(\Theta_n)$ at the basic wedge inclusions coincides with the localization of $\Psh(\Theta_n)$ at the basic spine inclusions.
\end{lem}
\begin{proof}
This is \cite[Section 13, footnote 2]{bsp}. The (straightforward) details are spelled out e.g. in \cite[Lemma 2.6]{campion-paste}. 
\end{proof}

\begin{Def}\label{def:theta-space}
Let $n \in \natsw$. We let $\Cat_n^\inc$ denote the localization of $\Psh(\Theta_n)$ at the basic spine inclusions. 
%
\end{Def}

\begin{rmk}\label{rmk:to-infty}
For $n \in \nats$, the $\infty$-category $\Cat_n^\inc$ is the the $\infty$-category of \defterm{flagged $(\infty,n)$-categories} in the sense of \cite{ayala-francis}. In the case $n = \omega$, we extend this terminology, referring to $\Cat_\omega^\inc$ as the \defterm{$\infty$-category of flagged $(\infty,\infty)$-categories}. Since the spine inclusions are manifestly preserved by the inclusion $\Psh(\Theta_n) \to \Psh(\Theta_{n+1})$, it results that the limit expression of \cref{rmk:colim} reflects down to exhibit $\Cat_\omega^\inc$ as the limit of the forgetful functors $\Cat_{n+1}^\inc \to \Cat_n^\inc$. It follows that $\Cat_\omega^\inc$ admits a description analogous to the description of $\Cat_n^\inc$ given in \cite{ayala-francis} for $n \in \nats$. Equivalently, $\Cat_\omega^\inc$ is the colimit in $\Pr^L$ of the inclusion functors $\Cat_n^\inc \to \Cat_{n+1}^\inc$.

We wish to emphasize that although the term ``the $(\infty,1)$-category of $(\infty,\infty)$-categories" is ambiguous (and we shall discuss inductive and coinductive equivalences presently -- see \cref{def:rezkloc} and \cref{def:coindloc}), these subtleties only arise later -- they are absent in the flagged setting of \cref{def:theta-space}, where we have Segal conditions but no Rezk completeness conditions. Thus the term ``the $(\infty,1)$-category of flagged $(\infty,\infty)$-categories" refers unambiguously to $\Cat_\omega^\inc$.\footnote{Of course, $\Cat_\omega^\inc$ has a nontrivial automorphism group $(\ints/2)^\omega$, but that is a different kind of ambiguity.}
\end{rmk}

\begin{Def}\label{def:nerve}
    Let $n \in \natsw$. The inclusion $\Theta_n \to \sCat_n$ induces a nerve functor $\nerve : \sCat_n \to \Psh(\Theta)$, which in fact lands in $\Cat_n^\inc$, since the spine inclusions sketch colimits which are in fact colimits in $\sCat_n$.
\end{Def}

\begin{rmk}
    Because the spine inclusions are closed under $\tilde \Sigma$, the suspension / hom adjunctions of \cref{def:susp} all descend from $\Psh(\Theta_n)$ to $\Cat_n^\inc$. We denote the induced functor 
    \[\Sigma : \Cat_n^\inc \to \Cat_{1+n}^\inc\]
    for $n \in \nats_\infty$. Similarly we have $\Sigma_/, \Sigma_{//}$. The nerve functor of \cref{def:nerve} carries suspensions to suspensions (see e.g. \cite[Section 2]{campion-paste}), so there is no risk of confusion with the notation of \cref{def:susp-strict}.
\end{rmk}

\begin{Def}
    We denote by $\globe_n = \Sigma^n [0] \in \Theta_n$ the $n$-fold suspension of a point. We continue to write $\globe_n$ for the same object in $\sCat_n$, or in $\Psh(\Theta_n)$ (via the Yoneda embedding), or in $\Cat_n^\inc$ (via the nerve functor). We write $\partial \globe_n \in \sCat_{n-1}$ for the $n-1$-fold suspension of $\partial[1]$, and continue to write $\partial \globe_n$ for the image of $\partial \globe_n$ under the nerve functor -- i.e. as an object of $\Psh(\Theta_m)$ or $\Cat_m^\inc$ for $n \leq m \leq \omega$.
\end{Def}

\subsection{Localizations of flagged $(\infty,n)$-categories}\label{subsec:loc}
In this subsection, we localize $\Cat_\omega^\inc$ to obtain various $\infty$-categories of higher categories. We consider four-parameter family of localizations, specified by the \defterm{category level} (\cref{def:cat-lev}), the \defterm{truncation multi-level} (\cref{def:truncloc}), the \defterm{Rezk completeness level} (\cref{def:rezkloc}), and the \defterm{coinductive completeness level} (\cref{def:coindloc}) (a further localization beyond Rezk completeness). Each is specified by a \defterm{level} -- an element of $\Levw = \ints_{\geq -2} \cup \{\omega\}$, except for the truncation level, which is specified by a \defterm{multi-level}, which gives a level for each dimension. This family of localizations includes e.g. $\Cat_n$ or $\Cat_{(m,n)}$, but also (via the delooping hypothesis) $\Alg_{E_k}(\Cat_n)$. Of the three families of examples mentioned in the previous sentence, the first two but not the last will be seen to admit Gray tensor products in \cref{sec:gray}. 

\begin{Def}\label{def:cat-lev}
Let $n \in \nats$. Letting $p : [1] \to [0]$ be the canonical projection, we consider the maps 
\[\Sigma^n p : \globe_{n+1} \to \globe_n\]
for $n \in \nats$.
\end{Def}

\begin{eg}
Let $m \leq n \leq \omega$. Then $\Cat_m^\inc$ is canonically identified, via the inclusion functor, with the localization of $\Cat_n^\inc$ at the maps $\Sigma^k p$ for $k \leq m < n$.
\end{eg}

\begin{Def}\label{def:truncloc}
Let $n \in \nats$ and $d \in \Lev$. We denote by 
\[\Sigma^n (t_d : S^{d+1} \to [0])\]
the \defterm{$n$th $d$-truncation map}.

Let $n \in \natsw$, and let $\underline{d} : [n] \to \Levw$ be a function (such a pair we call a \defterm{multi-level}). We denote by 
\[\Cat_{\underline{d}}^\inc\]
the localization of $\Cat_n^\inc$ at the morphisms $\{\Sigma^m t_{d(m)} \mid m \in [n]\}$.
\end{Def}

\begin{rmk}
Let $\underline d : [n] \to \Levw$ be a multi-level. Informally, $\Cat_{\underline{d}}^\inc$ comprises those flagged $(\infty,n)$-categories $C$ such that for all $m \in \nats$, and for all boundary shapes $\partial \globe_m \to C$, the space of $m$-cells with the given boundary is $d(m)$-truncated. 
\end{rmk}

\begin{rmk}
Let $\underline{d} : [n] \to \Levw$ be a multi-level. Then localizing $\Cat_{\underline{d}}^\inc$ at the maps $\Sigma^k p$ for $k \geq m$ results in a category which is canonically equivalent to $\Cat_{\underline{d}|_{[m]}}^\inc$.
\end{rmk}

\begin{rmk}
    Let $n \in \nats_\infty$. Note that both $[n]$ and $\Levw$ have natural orderings; a multi-level $\underline d : [n] \to \Levw$ is said to be \defterm{order-preserving} if it respects these orderings. In general, a multi-level is not required to be order-preserving. However, we shall see in \cref{sec:gray} that $\Cat_{\underline d}^\inc$ admits a Gray tensor product if and only if $\underline d$ is order-preserving.
\end{rmk}



\begin{obs}
Let $n \in \natsw$, let $(\underline{d}_i : [n] \to \Levw)_{i \in I}$ be a family of multi-levels, and let $\underline{d} : \nats \to \Levw$ be the infinum of this family (which is a minimum if the family is nonempty). Then $\cap_{i \in I} \Cat_{\underline{d}_i}^\inc = \Cat_{\underline{d}}^\inc$.
\end{obs}

\begin{eg}
Let $\const_\omega^n : [n] \to \Levw$ be the constant multi-level with value $\omega$. Then $\Cat_{\const_\omega^n}^\inc = \Cat_n^\inc$. In particular, $\Cat_{\const_\omega^\omega}^\inc = \Cat_\omega^\inc$.
\end{eg}

\begin{eg}
Let $C \in \Cat_\omega^\inc$, and let $n \in \nats$. Then the map $\Sigma^n t_{-2} : \partial \globe_n \to \globe_n$ is the canonical inclusion, and $C$ is $\Sigma^n t_{-2}$-local if and only if its spaces of $n$-morphisms-with-fixed-boundary are all contractible.

Let $n, m \in \nats$, and let $\underline{d} : [m+n] \to \Levw$ be a multi-level such that $\underline{d}(k) = -2$ for $0 \leq k < n$. Let $\underline{e} : [m] \to \Levw$ be the multi-level defined by $\underline{e}(k) = \underline{d}(k+n)$. Then by the delooping hypothesis, $\Cat_{\underline{d}}^\inc$ should be identifiable with $E_{n+1}$-monoids in $\Cat_{\underline{e}}^\inc$.

Note that these examples are only interesting if $\underline{d}$ is \emph{not} order-preserving, and thus we shall see (\cref{prop:susp-pres}) that the Gray tensor product does \emph{not} reflect straightforwardly to these settings.
\end{eg}

\begin{eg}
Let $C \in \Cat_\omega^\inc$, and let $n \in \nats$. The map $\Sigma^n t_{-1} : \partial \globe_{n+1} \to \globe_n$ is the canonical projection, and $C$ is $\Sigma^n t_{-1}$-local if and only if its $n$-fold hom-categories are all empty or contractible. For instance, the category of preorders is $\Cat^\inc_{0,-1}$.
\end{eg}

\begin{eg}
We have that $C$ is $\Sigma^n t_{0}$-local if and only if its $n$-fold hom-spaces are all discrete.

The nerve functor induces an equivalence $\nerve : \sCat_n \to \Cat_{\const_0^n}^\inc$, where $\const_0^n : [n] \to \Levw$ is constant with value $0$.
\end{eg}

\begin{eg}\label{eg:mninc}
Let $m,n \in \nats$, and let $\underline{d} : [n] \to \Levw$ be the multi-level defined by $\underline{d}(k) = m + n - k$. Then $\Cat_{\underline{d}}^\inc$ is the $\infty$-category of flagged $(m+n, n)$-categories.
\end{eg}

\begin{Def}\label{def:rezkloc}
Let $I \in \sCat_1$ denote the walking isomorphism and $r \in \nats$.  We denote by 
\[\Sigma^r (\rho  : I \to [0])\]
the \defterm{$r$th Rezk map}. For $n \in \natsw$, we denote by $\Cat_n$ the localization of $\Cat_n^\inc$ at all of the Rezk maps. More generally, for $R \in \natsw$, we denote by 
\[\Cat_n^{\inc < R}\]
the localization of $\Cat_n^\inc$ at the maps $\Sigma^r \rho$ for $r \geq R$. For $R = 0$, we denote $\Cat_{\underline{d}} =  \Cat_n^{\inc < 0}$. In particular, $\Cat_n^{\inc < 0} = \Cat_n$ and $\Cat_n^{\inc < \omega} = \Cat_n^\inc$.

More generally, for $\underline{d}: [n] \to \Levw$ a multi-level, we denote by 
\[\Cat_{\underline{d}}^{\inc < R} = \Cat_{\underline{d}}^\inc \cap \Cat_n^{\inc < R}\]
the intersection.
\end{Def}

\begin{rmk}
Let $n, R \in \natsw$. The $\infty$-category $\Cat_n^{\inc < R}$ is an $\infty$-category of flagged $(\infty,n)$-categories which are Rezk-complete in dimensions $\geq R$, so that the flagging in higher dimensions is the canonical one.
\end{rmk}

\begin{eg}
let $n \in \nats$. Then $\Cat_n$ is canonically equivalent to the $\infty$-category of $(\infty,n)$-categories in the usual sense. And $\Cat_\omega$ is the limit of these along the forgetful functors. It is the $\infty$-category of $(\infty,\infty)$-categories (with inductive equivalences).
\end{eg}

\begin{eg}
Let $n \in \natsw$. Then $\Cat_{\const_0^n} = \sCat_n \cap \Cat_n$ is canonically equivalent to the category of \defterm{gaunt} $n$-categories.
\end{eg}

\begin{eg}
Let $m, n \in \nats$, and let $\underline{d}$ be defined as in \cref{eg:mninc}. Then $\Cat_{\underline{d}}$ is the $\infty$-category of $(m+n,n)$-categories in the usual sense.
\end{eg}

\begin{eg}
$\Cat_{0,-1}$ is the strict $1$-category of posets.
\end{eg}

\begin{Def}\label{def:coindloc}
Let $n \in \natsw$. Let $J$ be the walking coinductive equivalence (\cite[Def 4.20]{henry-loubaton}). Then for $c \in \nats$, the map 
\[\Sigma^c (\gamma : J \to [0])\]
is the \defterm{$c$th coinduction map}. For $C \in \nats$, the localization of $\Cat_n^\inc$ at the maps $\Sigma^c \gamma$ for $c \geq C$ is denoted 
\[\Cat_n^{\inc <^\coind C}.\]
For $C = 0$, we denote $\Cat_n^{\coind} = \Cat_n^{\inc <^\coind 0}$. More generally, for a multi-level $\underline d : [n] \to \Levw$, we denote $\Cat_{\underline d}^{\inc <^\coind C} = \Cat_{\underline d}^\inc \cap \Cat_n^{\inc <^\coind C}$, and for $C = 0$ we denote $\Cat_{\underline d}^\coind = \Cat_{\underline d}^{\inc <^\coind 0}$.
\end{Def}

\begin{eg}
$\Cat_\omega^{\inc <^\coind 0}$ is the $\infty$-category of $(\infty,\infty)$-categories with coinductive equivalences.
\end{eg}

\begin{obs}
The distinction between inductive and coinductive equivalences only arises when the category level is infinite. For $n, C \in \nats$, we have $\Cat_n^{\inc <^\coind C} = \Cat_n^{\inc < C}$; we will by default write $\Cat_n^{\inc < C}$ for this $\infty$-category, but we allow $\Cat_n^{\inc <^\coind C}$ notationally. Similarly, when $n$ is finite we have $\Cat_{\underline d}^{\inc <^\coind C} = \Cat_{\underline d}^{\inc < C}$ for any multi-level $\underline d : [n] \to \Levw$.
\end{obs}

\begin{rmk}
    In principle, all four parameters introduced here -- the category level, truncation level, Rezk completeness level, and coinductive completeness level -- could be indexed by a multi-level rather than just a level. The results of \cref{sec:sites} would also go through here. We do not rule out that these more general localizations might be interesting, but they will not be exponential ideals for the Gray tensor product except in the cases we have in fact considered, where the class of maps we localize at is stable under suspension (\cref{prop:susp-pres}).
\end{rmk}

\section{Large sites for $(\infty,n)$-categories}\label{sec:sites}
In this section, we introduce the notion of a \defterm{suitable site} $\calS$ (\cref{def:suitable}) of torsion-free complexes for modeling $(\infty,\infty)$-categories. For a suitable site $\calS$, we describe generators (\cref{def:j-square}) for a localization $L_\calS \Psh(\calS)$ of $\Psh(\calS)$, and we show (\cref{cor:locn}) that the canonical nerve $\Cat_\omega^\inc \to \Psh(\calS)$ induces an equivalence $\Cat_\omega^\inc \simeq L_\calS \Psh(\calS)$. We deduce (\cref{cor:locn'}) explicit presentations for the further localizations of $\Cat_\omega^\inc$ discussed in \cref{subsec:loc} above as presheaves of spaces on $\calS$.

The results of this section should be seen as similar in spirit to the work in \cite{bsp} on presheaves on $\Upsilon_n$, or to the work in \cite{henry-regular} on presheaves on regular polygraphs, in that we model $(\infty,n)$-categories as presheaves on a relatively large site of computads. A notable feature is that in this setting, there are no ``spine-filling" conditions -- the ``Segal conditions" here really are just the cell-gluing conditions of \cref{def:j-square}.

This section is short, but new: all of the real work in \cref{cor:locn} is done in \cite{campion-paste}.

\begin{Def}\label{def:suitable}
A \defterm{suitable site} $\TPar$ is a a set of finite torsion-free complexes \cite{forest}, closed under subcomplexes, such that $\Theta$ is contained in $\TPar$.
\end{Def}

\defterm{Torsion-free complexes} are a restricted class of computads (i.e. free strict $\omega$-categories) similar to Street's \defterm{parity complexes} \cite{street}. We refer to \cite{forest} for details, and to \cite[Secion 1]{campion-paste} for a brief overview. Here let it suffice to say that every loop-free augmented directed complex (\cite{steiner}) is a torsion-free complex \cite[Thm 3.4.4.22]{forest-thesis}, and in particular every object of $\Theta$, each of Street's orientals \cite{street}, and each Gray cube is a torsion-free complex, as are all subcomplexes of these (a subcomplex is a subcategory whose inclusion sends free generators to free generators).

\begin{eg} $~$
\begin{itemize}
    \item The category $\TPar$ of all torsion-free complexes is a suitable site.
    \item The category of loop-free Steiner complexes is a suitable site \cite[Thm 3.4.4.22]{forest-thesis}.
    \item The category of strongly loop-free Steiner complexes is a suitable site.
    \item The closure of $\Theta$ under subcomplexes is a suitable site.
    \item The closure under subcomplexes of the union of $\Theta$ and the lax Gray cubes is a suitable site.
\end{itemize}
And so forth.
\end{eg}

\begin{rmk}\label{rmk:handbasket}
It would be desirable to weaken the assumption that $\Theta \subseteq \TPar$ to the assumption that $\Theta$ is contained in the idempotent completion of $\TPar$, so as to include the case where $\TPar$ comprises the Gray cubes (cf. \cite{campion-dense}). We are not sure if this is possible, ultimately in part because we do not know whether torsion-free complexes are closed under retracts. Note that computads are closed under retracts \cite[Thm 7.1]{metayer}. 
\end{rmk}

\begin{Def}\label{def:j-square}
Let $\TPar$ be a suitable site. Let $J_\TPar \subset \Mor(\Psh(\TPar))$ comprise, for each $G \hookleftarrow S \rightarrowtail A$ in $\TPar$ such that $A \rightarrowtail S$ is a monomorphism, $S \hookrightarrow G = \partial \globe^m \to \globe^m$ is a generating folk cofibration, and $A \ast_S G$ is in $\TPar$, the morphism 
\[\yo A \cup_{\yo S} \yo G \to \yo(A \ast_S G).\]
We write $L_\TPar \Psh(\TPar)$ for the localization with respect to $J_\TPar$, and $L_\TPar$ the localization functor.
\end{Def}

\begin{prop}\label{prop:i-cube-good}
Let $\TPar$ be a suitable site. Every object of $\Cat_\omega^\inc$ is local with respect to $J_\TPar$.
\end{prop}
\begin{proof}
Each morphism of $J_\TPar$ sketches a certain pushout. The main result of \cite{campion-paste} tells us that this is also a pushout in $\Cat_\omega^\inc$.
\end{proof}

\begin{prop}\label{prop:strong}
Let $\TPar$ be a suitable site. In $L_\TPar \Psh(\TPar)$, the globe category $\globe$ is a colimit-generator. 
\end{prop}
\begin{proof}
The localization is subcanonical - the representables generate everything under colimits. Moreover, the morphisms of $J_\TPar$ ensure that every representable becomes an iterated colimit of globes in $L_\TPar \Psh(\TPar)$. Thus the globes are a colimit-generator.
\end{proof}

\begin{lem}\label{lem:criteria}
Let $A \subseteq \Gaunt$. Suppose that $\Theta$ is contained in the idempotent completion of $A$, so that $A$ is dense in $\Cat_\omega^\inc$, and we may regard $\Cat_\omega^\inc$ canonically as a full subcategory of $\Psh(A)$. Let $L \Psh(A)$ be an accessible localization of $\Psh(A)$, with localization functor $L$. Suppose that $\Cat_\omega^\inc \subseteq L \Psh(A)$, and that $\Theta$ generates $L \Psh(A)$ under colimits. Then the following are equivalent:
\begin{enumerate}
    \item $L\Psh(A) = \Cat_\omega^\inc$;
    \item The restriction functor $\iota^\ast : \Psh(A) \to \Psh(\Theta)$ carries generating $L$-acyclic maps to $L_{\Cat_\omega^\inc}$-acyclic maps, and for every $X \in L\Psh(A)$, the restriction $\iota^\ast X$ of $X$ to a presheaf on $\Theta$ lies in $\Cat_\omega^\inc \subset \Psh(\Theta)$.
\end{enumerate}
\end{lem}
\begin{proof}
$(1) \Rightarrow (2)$ is trivial.

$(2) \Rightarrow (1)$: Note that the restriction functor $\iota^\ast : \Psh(A) \to \Psh(\Theta)$ has both a left adjoint $\iota_!$ and a right adjoint $\iota_\ast$. By hypothesis, $\iota^\ast$ restricts to a functor $i^\ast: L\Psh(A) \to L_{\Cat_\omega^\inc} \Psh(\Theta)$. Moreover, by hypothesis $\iota^\ast$ carries $L_{\Cat_\omega^\inc}$-acyclic maps to $L$-acyclic maps, so by adjunction $\iota_\ast$ also restricts to a functor $i_\ast : L_{\Cat_\omega^\inc} \Psh(\Theta) \to L\Psh(A)$, right adjoint to the fully faithful $i^\ast$. Now the essential image of $i^\ast$ is closed under colimits and contains $\Theta$. By hypothesis, $\Theta$ generates $L\Psh(A)$ under colimits, so it is all of $L\Psh(A)$. Thus $i^\ast$ is essentially surjective, and hence an equivalence as desired.
\end{proof}


\begin{thm}\label{cor:locn}
Let $\TPar$ be a suitable site. Then $L_\TPar \Psh(\TPar) = \Cat_\omega^\inc$.
\end{thm}
\begin{proof}
We verify the hypotheses of \cref{lem:criteria}. As noted, we have $\Cat_\omega^\inc \subseteq L_\TPar \Psh(\TPar)$. $\Theta$ is a colimit-generator in $L_{\TPar}\Psh(A)$ by \cref{prop:strong}. So it remains to check that the restriction functor $\iota^\ast : \Psh(\TPar) \to \Psh(\Theta)$ carries $J_\TPar$ to $L_{\Cat_\omega^\inc}$-acyclic morphisms and carries objects of $L\Psh(\TPar)$ to flagged $(\infty,\infty)$-categories.

As $\iota^\ast$ preserves colimits, it carries the morphisms of $J_\TPar$ in $\Psh(A)$ to their counterparts in $\Psh(\Theta)$, which are $L_{\Cat_\omega^\inc}$-acyclic by the main result of \cite{campion-paste}.

To see that $\iota^\ast$ carries any $X \in L\Psh(A)$ to a flagged $(\infty,\infty)$-category, it suffices to observe that the spine inclusions (which generate the localization for $\Cat_\omega^\inc$) are easily seen (by induction on skeleta) to be $L_\TPar$-acyclic.
\end{proof}

\begin{cor}\label{cor:locn'}
Let $\TPar$ be a suitable site. Let $n \in \natsw$, let $\underline{d} : [n] \to \Levw$ be a multi-level, and let $R, C \in \natsw$. Then $\Cat_{\underline{d}}^{\inc < R}$ and $\Cat_{\underline{d}}^{\inc <^\coind C}$ are canonically equivalent to the localizations of $\Psh(\TPar)$ at $J_\TPar$ plus the appropriate maps from \cref{def:rezkloc} and \cref{def:coindloc}.
\end{cor}

\begin{eg}
    Let $n \in \natsw$ and $\calS$ a suitable site. Then $\Cat_n = \Cat_n^{\inc < 0}$ is the localization of $\Psh(\calS)$ at the maps of \cref{def:j-square} and the maps of \cref{def:rezkloc} with $R = 0$. This is a new model for the $\infty$-category of $(\infty,n)$-categories. Similarly, $\Cat_\omega^\coind = \Cat_\omega^{\inc <^\coind 0}$ is the localization of $\Psh(\calS)$ at the maps of \cref{def:j-square} and the maps of \cref{def:coindloc} with $C = 0$. The additional flexibility in \cref{cor:locn'} allows us to also specify truncation levels, or to impose Rezk completeness or coinductive completeness only above some category level.
\end{eg}


\section{The Gray tensor product}\label{sec:gray}

In this section, we reap the benefits of the new models built in \cref{sec:sites} by constructing the Gray tensor product of $(\infty,n)$-categories for arbitrary $n$. We begin in \cref{subsec:gray0} by constructing the Gray tensor product on $\Cat_\omega^\inc$ by introducing the notion of a \defterm{monoidally suitable site} $\calS$ (\cref{def:mon-suit}) and observing that at least one such site exists (\cref{eg:mon-suit}). A monoidally suitable site is by definition closed under the strict Gray tensor product, so we Day convolve it up (\cref{thm:day-conv}) and then reflect it down (\cref{thm:day-refl}). In the latter step we are able to directly deduce that $\Cat_\omega^\inc$ is an exponential ideal in $\Psh(\calS)$ by appeal to a result of \cite{ara-lucas} in the strict $\omega$-categorical literature, thanks to the explicit control we have over the localization from \cref{def:j-square}. Then in \cref{subsec:gray1}, we reflect the Gray tensor product further to some, but not all, of the $\infty$-categories of higher $\infty$-categories considered in \cref{subsec:loc}. Indeed, \cref{prop:susp-pres} gives a restriction on which localizations will admit a Gray tensor product: the acyclic maps must be closed under suspension. In \cref{thm:gray-reflect} and \cref{cor:gray-reflect} we see that this is the only restriction for localizations in the 4-parameter family of \cref{subsec:loc}. In other words, a localization of the form $\Cat_{\underline d}^{\inc < R}$ or $\Cat_{\underline d}^{\inc <^\coind C}$ is an exponential ideal in $\Cat_\omega^\inc$ if and only if the multi-level $\underline d$ is order-preserving. 

\subsection{The Gray tensor product on $\Cat_\omega^\inc$}\label{subsec:gray0}

\begin{Def}\label{def:mon-suit}
Let $\TPar \subset \sCat_\omega$ be a suitable site. We say that $\TPar$ is \defterm{monoidally suitable} if it is closed under the Gray tensor product on $\sCat_\omega$.
\end{Def}

\begin{eg}\label{eg:mon-suit}
The collection of \defterm{strongly loop-free complexes} in the sense of Steiner is a monoidally suitable site \cite{steiner}.
\end{eg}

\begin{thm}\label{thm:day-conv}
Let $\TPar$ be a monoidally suitable site. There is a monoidal biclosed structure $\hat \otimes$ on $\Psh(\TPar)$ such that the Yoneda embedding $\yo : \TPar \to \Psh(\TPar)$ is strong monoidal. Moreover, for any cocomplete monoidal $\infty$-category $\calE$ with colimits preserved separately in each variable, $\yo$ induces an equivalence $\Fun_{E_1}^{\lax,\cocts}(\Psh(\TPar),\calE) \to \Fun_{E_1}^\lax(\TPar,\calE)$.
\end{thm}
\begin{proof}
This is \cite[Proposition 4.8.1.10]{ha}, applied in the case where $\calK = \emptyset$, $\calK'$ consists of all small $\infty$-categories, and $\calO$ is the $E_1$ operad, and $\calC = \TPar$, considered as a monoidal category under the Gray tensor product $\otimes$. (Strictly speaking, we deduce from here the above statement amended to say that $\hat \otimes$ preserves colimits separately in each variable, and we deduce that $\Psh(\square)$ is monoidal biclosed by the $\infty$-categorical adjoint functor theorem (\cite[Corollary 5.5.2.9]{htt}).)
\end{proof}

\begin{thm}\label{thm:day-refl}
Let $\TPar$ be a monoidally suitable site. The fully faithful nerve $\Cat_\omega^\inc \to \Psh(\TPar)$ exhibits $\Cat_\omega^\inc$ as an exponential ideal in the monoidal category $\Psh(\TPar)$ of \cref{thm:day-conv}. 
Thus the localization $L_\TPar$ is strong monoidal. Moreover, for any cocomplete monoidal $\infty$-category $\calE$ with colimits preserved by $\otimes$ separately in each variable, precomposition with $L_\TPar$ induces a fully faithful functor
\[
\Fun_{E_1}^{\strong,\cocts}(\Cat_\omega^\inc, \calE) \to \Fun_{E_1}^{\strong, \cocts}(\Psh(\TPar), \calE)
\]
whose essential image comprises those strong monoidal, colimit preserving functors $\Psh(\TPar) \to \calE$ which carry the morphisms of $J_\TPar$ to equivalences.
\end{thm}
\begin{proof}
By \cite[Proposition 4.1.7.4]{ha}, it suffices to show that the $E_1$ monoidal structure $\hat \otimes$ is compatible with the localization functor $L_\TPar$. That is, we must show that if $f : \yo G \cup_{\yo S} \yo A \to \yo(G \ast_S A) \in J_\TPar$, then $L_\TPar(f \hat \otimes \id_Z)$ and $L_\TPar(\id_Z \hat \otimes f)$ are equivalences for each presheaf $Z$. We check the first statement; the proof in the other case is similar. Since $L_\TPar$ commutes with colimits and $\Psh(\TPar)$ is generated under colimits by representables, it suffices to check this when $Z = \yo z$ is representable. Since $\TPar$ is closed under $\otimes$, $f \hat \otimes \id_{\yo z}$ is the map $\yo (G \otimes z) \cup_{\yo(S \otimes z)} \yo(A \otimes z) \to \yo(G \otimes z \ast_{S \otimes z} A \otimes z)$. By \cite{ara-lucas}, $(-) \otimes z$ preserves monic pushouts along folk cofibrations, so this map is $L_\TPar$-acyclic as desired.
\end{proof}

\begin{cor}\label{cor:day-refl}
Let $\TPar$ be a moniodally suitable site. The fully faithful functor $\TPar \to \Cat_\omega^\inc$ is strong monoidal. Here $\TPar$ is regarded as a monoidal ($\infty$-)category under the usual Gray tensor product $\otimes$ of strict $\omega$-categories, and $\Cat_\omega^\inc$ is regarded as monoidal under the Gray tensor product $\otimes$ of \cref{thm:day-refl}. Moreover, we have the following two universal properties:
\begin{enumerate}
    \item For any monoidal biclosed, cocomplete $\infty$-category $\calE$, composition with $\TPar \to \Cat^\inc_\omega$ induces a fully faithful functor 
    \[\Fun_{E_1}^{\strong, \cocts}(\Cat^\inc_\omega, \calE) \to \Fun_{E_1}^\strong(\TPar, \calE)\]
    whose essential image comprises those monoidal functors $\TPar \to \calE$ whose underlying functor preserves monic pushouts along folk cofibrations.
    \item $\otimes$ is the unique monoidal biclosed structure on $\Cat_\omega^\inc$ such that $\TPar \to \Cat_\omega^\inc$ is strong monoidal.
\end{enumerate}
\end{cor}
\begin{proof}
We can factor the functor as $\TPar \xrightarrow{\yo} \Psh(\TPar) \xrightarrow{L_\TPar} \Cat_\omega^\inc$.
The first factor is strong monoidal by \cref{thm:day-conv} and the second factor is strong monoidal by \cref{thm:day-refl}. The first universal property follows by concatenating \cref{thm:day-conv} and \cref{thm:day-refl}. For the second, suppose that $F : \TPar \to \Cat_\omega^\inc$ is strong monoidal with respect to another monoidal biclosed structure $\otimes'$ on $\Cat_\omega^\inc$, and that its underlying functor is the usual $\TPar \to \Cat_\omega^\inc$. By the first universal property, there is an induced cocontinuous, monoidal biclosed functor $\hat F : \Cat_\omega^\inc \to \Cat_\omega^\inc$ which is strong monoidal from $\otimes$ to $\otimes'$, and whose restriction to $\TPar$ is the identity. By density, $\hat F$ is naturally isomorphic to the identity. So $\hat F$ is an equivalence. Since $\hat F$ is strong monoidal, it is a monoidal equivalence.
\end{proof}

\subsection{The Gray tensor product of $(\infty,n)$-categories}\label{subsec:gray1}
We now study which of the localizations from \cref{subsec:loc} are exponential ideals with respect to the Gray tensor product of \cref{thm:day-refl}.

\begin{lem}\label{lem:decomp}
Let $C \in \Cat_\omega^\inc$. Then there is a colimit decomposition in $\Cat_\omega^\inc$: 
$$\square^1 \otimes (\Sigma C) = (((\Sigma C) \vee \square^1) \cup_{\partial \square^1} (\square^1 \vee (\Sigma C))) \cup_{\Sigma (\partial \square^1 \otimes C)} \Sigma (\square^1 \otimes C)$$
This decomposition is natural in $C$. Moreover, if $C \in \TPar$ for a monoidally suitable site $\TPar$, then these pushouts may be computed either weakly or strictly.
\end{lem}
\begin{proof}
Let $\TPar$ be a monoidally suitable site. First consider the case where $C \in \TPar$. By results in \cite[Secion 2]{campion-paste}, the wedge sums $(\Sigma C) \vee \square^1$ and $\square^1 \vee (\Sigma C)$ may be computed either strictly or weakly. Then because the pushout $((\Sigma C) \vee \square^1) \cup_{\partial \square^1} (\square^1 \vee (\Sigma C))$ is computed as a pushout of in $\Psh(\Theta)$, it also agrees when computed strictly or weakly. So this factor may be computed strictly. We see that it accounts for all cells in $\square^1 \otimes (\Sigma C)$ of the form $x \otimes y$ where $x$ or $y$ is of dimension $0$, and all composites thereof. The $\Sigma(C \otimes \square^1)$ then accounts for only cells in the ``long" / ``diagonal" hom-category, filling in those cells of the form $(\Sigma x) \otimes (\Sigma y)$. Moreover, the pushout attaching this part is a pushout in $\Psh(\Theta)$, so it can be computed either strictly or weakly. Thus the formula is correct when $C \in \TPar$, and may be computed either weakly or strictly.

All of the maps involved are natural in $C$, and every part of the expression preserves contractible colimits. So after checking that the two sides are equivalent when $C = \emptyset$, we deduce that the relevant maps exist by extending by contractible colimits. We deduce that the map between the two sides is an equivalence by reducing to the case where $C = \globe^n$ is a globe. Here it is true by the previous paragraph.
\end{proof}

\begin{Def}
We write $\square^1 \otimes^\funny \Sigma C = ((\Sigma C) \vee \square^1) \cup_{\partial \square^1} (\square^1 \vee (\Sigma C))$, so that \cref{lem:decomp} reads $\square^1 \otimes (\Sigma C) = (\square^1 \otimes^\funny (\Sigma C)) \cup_{\Sigma (\partial \square^1 \otimes C)} \Sigma (\square^1 \otimes C)$. This is an instance of the so-called \defterm{funny tensor product} of strict $\omega$-categories, which we will not develop further right now.
\end{Def}

\begin{lem}\label{lem:susp-tensor}
For $C \in \Cat_\omega^\inc$, we have:
\begin{equation*}
    \Sigma C = \square^1 \otimes C \cup_{\partial \square^1 \otimes C} \partial \square^1.
\end{equation*}
\end{lem}
\begin{proof}
When $C$ is a strict $\omega$-category, we have such a strict pushout square, natural in $C$ \cite[Cor. B.6.6]{ara-maltsiniotis}. In particular, we have such a natural strict pushout when $C \in \TPar$, for a monoidally suitable site $\TPar$. This extends also to $C = \emptyset$, the empty presheaf. As every term in the commutative square preserves contractible colimits in $C$ (see \cite[Section 1]{campion-paste} for $\Sigma$), and as the density presentation of any $X \in \Psh(\TPar)$ with respect to $\TPar^\triangleleft$ is contractible, we obtain a natural commutative square $\partial \square^1 \hat \otimes X \rightrightarrows \partial \square^1 , \square^1 \hat \otimes X \rightrightarrows \tilde \Sigma X$ for all $X \in \Psh(\TPar)$, and in particular for $X \in \Cat_\omega^\inc$.

To see that the square is a homotopy pushout, it suffices to check on a colimit-generating set, such as the globes $C = \globe^n$. We proceed by induction on $n$. When $n = 0$, the statement is trivial. Inductively, write $\globe^{n+1} = \Sigma G$, where $G = \globe^n$. We have a commutative diagram:
\begin{equation*}
\begin{tikzcd}
& \partial \square^1 \otimes \Sigma G \ar[r] \ar[d] & 
\Sigma \emptyset \ar[d] \\
\Sigma(\partial \square^1 \otimes G) \ar[r] \ar[d] & 
\square^1 \otimes^\funny \Sigma G \ar[r] \ar[d] & 
\Sigma(\partial \square^1) \ar[d] \\
\Sigma(\square^1 \otimes G) \ar[r] & 
\square^1 \otimes \Sigma G \ar[r] & 
\Sigma^2 G
\end{tikzcd}
\end{equation*}
We wish to show that the vertical composite of the two squares on the right is a pushout. We claim first that the top-right square is a pushout. This follows by commuting colimits with colimits:
\begin{align*}
    (\square^1 \otimes^\funny \Sigma G) \cup_{\Sigma G \amalg \Sigma G} (\square^0 \amalg \square^0)
    &= (((\Sigma G) \vee \square^1) \cup_{\partial \square^1} (\square^1 \vee (\Sigma G))) \cup_{\Sigma G \amalg \Sigma G} (\square^0 \amalg \square^0) \\
    &= (((\Sigma G) \vee \square^1) \cup_{\partial \square^1} (\square^1 \vee (\Sigma G))) \cup_{\Sigma G \cup_\emptyset \Sigma G} (\square^0 \cup_\emptyset \square^0) \\
    &= (((\Sigma G) \vee \square^1) \cup_{\Sigma G} \square^0) \cup_{\partial \square^1 \cup_\emptyset \emptyset} (\square^1 \vee (\Sigma G)) \cup_{\Sigma G} \square^0) \\
    &= (\square^0 \vee \square^1) \cup_{\partial \square^1} (\square^1 \vee \square^0) \\
    &= \square^1 \cup_{\partial \square^1} \square^1 \\
    &= \Sigma(\partial \square^1)
\end{align*}
as desired.
So it will suffice to show that the bottom-right square is a pushout. The bottom-left square is a pushout by \cref{lem:decomp}. So it will suffice to show that the horizontal composite of the bottom two squares is a pushout. This square results by applying $\Sigma$ to a commutative square which by induction expresses the pushout $\Sigma G = \square^1 \otimes G \cup_{\partial \square^1 \otimes G} \partial \square^1$. As $\Sigma$ preserves contractible colimits, this is likewise a pushout as desired.
%
\end{proof}

\begin{prop}\label{prop:susp-pres}
Let $L : \Cat_\omega^\inc {}^\to_\leftarrow L\Cat_\omega^\inc : i$ be a localization. Suppose that $L \Cat_\omega^\inc$ is an exponential ideal in $\Cat_\omega^\inc$. Then the $L$-acyclic morphisms are closed under suspension.
\end{prop}
\begin{proof}
Let $f : A \to B$ be $L$-acyclic. By the exponential ideal condition, $\square^1 \otimes f : \square^1 \otimes A \to \square^1 \otimes B$ is $L$-acyclic. By cobase-change (from \cref{lem:susp-tensor}), $\Sigma f : \Sigma A \to \Sigma B$ is likewise $L$-acyclic.
\end{proof}

\begin{rmk}
Let $n \in \natsw$ and $\underline{d} : [n] \to \Levw$. By \cref{prop:susp-pres}, the localization $\Cat_{\underline{d}}^\inc$ is not compatible with the Gray tensor product unless $d(m)$ is weakly decreasing in $m$. We shall soon see (\cref{cor:gray-reflect}) that at least for this class of localizations, this is the \emph{only} restriction -- so long as $\underline{d}$ is weakly decreasing, $\Cat_{\underline{d}}^\inc$ is compatible with the Gray tensor product. Similarly, the localizations $\Cat_\omega^{\inc < R}$ and $\Cat_\omega^{\inc<^\coind C}$ (which are suspension-stable) will be seen to be compatible with the Gray tensor product.
\end{rmk}

\begin{lem}\label{lem:fun-acyc}
Let $\calC$ be presentable category. Let $F : \calC \to \calC$ be a functor preserving contractible colimits. Let $S \subseteq \Mor \calC$ be a set of morphisms, and let $F(S)$ be its image under $F$. If $f$ is $L_SS$-acyclic, it follows that $F(f)$ is $L_{F(S)}$-acyclic.
\end{lem}
\begin{proof}
Composites, isomorphisms, and retracts are preserved by all functors. Pushouts and transfinite composition are contractible colimits. The initial object in the arrow category is carried to an isomorphism, and nonempty coproducts in the arrow category are carried to (infinitary) pushouts under the image of the initial object. 
\end{proof}

\begin{lem}\label{lem:term-tensor-acyc}
Let $C \in \Cat_\omega^\inc$ and $t : C \to \square^0$ the unique morphism. Then $\square^1 \otimes t$ is $L_{t,\Sigma t}$-acyclic.
\end{lem}
\begin{proof}
By \cref{lem:susp-tensor}, the map $\Sigma t$ is a cobase-change of $\square^1 \otimes t$. Moreover, the relevant map $\square^1 \otimes \square^0 \to \Sigma \square^0$ is an isomorphism. So we have a factorization of of $\square^1 \otimes t$ as $\square^1 \otimes C \to \Sigma C \to \Sigma \square^0$. The second map is none other other than $\Sigma t$, which is tautologically $\Sigma t$-acyclic. The first map is a cobase-change of $\partial \square^1 \otimes C \to \partial \square^1$, which is $t$-acyclic by closure under coproducts. So by composition, $\square^1 \otimes t$ is likewise $L_{t, \Sigma t}$-acyclic.
\end{proof}

\begin{lem}\label{lem:susp-tensor-acyc}
Let $f: C \to D$ be a morphism in $\Cat_\omega^\inc$. Let $L_{\Sigma^\ast f}$ denote the localization of $\Cat_\omega^\inc$ at $f$ and all of its suspensions, and let $L_{\Sigma^{\ast+1} f}$ denote the localization at $\Sigma f$ and all of its suspensions. Suppose that $\square^1 \otimes f$ if $L_{\Sigma^\ast f}$-acyclic. Then $\square^1 \otimes (\Sigma f)$ is $L_{\Sigma^{\ast+1} f}$ acyclic.
\end{lem}
\begin{proof}
By \cref{lem:decomp}, $\square^1 \otimes (\Sigma f) = (((\Sigma f) \vee \square^1) \cup_{\partial \square^1} (\square^1 \vee (\Sigma f))) \cup_{\Sigma (\partial \square^1 \otimes f)} \Sigma (\square^1 \otimes f)$. By stability under pushouts, it will suffice to show that $\Sigma f$, $\Sigma(\partial \square^1 \otimes f)$, and $\Sigma (\square^1 \otimes f)$ are $L_{\Sigma^{\ast+1}}$-acyclic. For $\Sigma f$, this is tautological. For $\Sigma(\partial \square^1 \otimes f)$ and $\Sigma (\square^1 \otimes f)$, note that $\partial \square^1 \otimes f$ is $L_{\Sigma^\ast f}$-acyclic by closure under coproducts, that $\square^1 \otimes \square^1$ is $L_{\Sigma^\ast f}$-acyclic by hypothesis, and that the suspension of an $L_{\Sigma^\ast f}$-acyclic morphism is $L_{\Sigma^{\ast+1} f}$-acyclic by \cref{lem:fun-acyc}.
\end{proof}

\begin{thm}\label{thm:gray-reflect}
Let $L = L_S$ be a localization of $\Cat_\omega^\inc$. Suppose that $S$ is a union of sets of the form $\{\Sigma^n t_C \mid n \geq N(C)\}$ where $t_C : C \to \square^0$ is the unique map. Then the localization $L_S$ is compatible with the Gray tensor product. Thus there is a monoidal biclosed structure $\otimes_S$ on $L_S \Cat_\omega^\inc$ such that the localization $L_S : \Cat_\omega^\inc \to L_S \Cat_\omega^\inc$ is strong monoidal. Moreover, for any cocomplete monoidal $\infty$-category $\calE$ with tensor product preserving colimits separately in each variable, $L_S$ induces a fully faithful functor
\[
\Fun_{E_1}^{\strong,\cocts}(L_S\Cat_\omega^\inc, \calE) \to \Fun_{E_1}^{\strong, \cocts}(\Cat_\omega^\inc, \calE)
\]
whose image comprises those strong monoidal functors $\Cat_\omega^\inc \to \calE$ whose underlying functor factors through $L_S$.
\end{thm}
\begin{proof}
By \cite[Proposition 4.1.7.4]{ha}, it suffices to show that, for every $f \in S$, both $X \otimes f$ and $f \otimes X$ are $L_S$-acyclic. Because the cubes are dense in $\Cat_\omega^\inc$, it suffices to treat the case where $X = \square^n$ is a cube. Because the tensor product is associative, by induction on $n$ it suffices to treat the case where $X = \square^1$. By \cref{lem:term-tensor-acyc}, $\square^1 \otimes t_C$ is $L_{\Sigma^\ast t_C}$-acyclic. By \cref{lem:susp-tensor-acyc} and induction on $n$, we obtain that $\square^1 \otimes \Sigma^n t_C$ is $L_{\Sigma^{\ast + n} t_C}$-acyclic. So for $n \geq N$, $\square^1 \otimes \Sigma^n t_C$ is $L_{\Sigma^{\ast + N}}$-acyclic, and hence $L_S$-acyclic as desired. The argument for $\Sigma^n f \otimes \square^1$ is similar, using dual lemmas to the above.
\end{proof}

\begin{cor}\label{cor:gray-reflect}
Let $n \in \natsw$, let $\underline{d} : [n] \to \Levw$ a weakly decreasing multi-level, and let $R,C \in \natsw$. Then $\Cat_{\underline{d}}^{\inc < R}$ and $\Cat_{\underline{d}}^{\inc <^\coind C}$ are exponential ideals in $\Cat_\omega^\inc$, and hence admit a reflected Gray tensor product with the universal property of \cref{thm:gray-reflect}.
\end{cor}
\begin{proof}
Each of these localizations is of the form required by \cref{thm:gray-reflect}.
\end{proof}

\begin{eg}
Such $\infty$-categories as $\Cat_n$, $\sCat_n$, $\Cat_{(m+n, n)}$, for $m,n \in \natsw$ etc. all admit a monoidal biclosed Gray tensor product.
\end{eg}

\section{Uniqueness of the Gray tensor product}\label{sec:unique}
In \cref{sec:gray}, we constructed the Gray tensor product on $\Cat_\omega^\inc$ and on various localizations thereof. This construction came with some kind of universal property (\cref{thm:day-refl}), but it was dependent on the choice of a \defterm{monoidally suitable site} (\cref{def:mon-suit}). In this section, we give several universal properties of these Gray tensor products which do not refer to auxiliary data such as a monoidally suitable site. We first give (\cref{thm:gray-unique} and \cref{cor:gray-unique1}) universal properties which depend on the choice of a dense subcategory $\calS \subseteq \Cat_\omega^\inc$ closed under the strict Gray tensor product -- a slightly weaker condition than being a monoidally suitable site. By \cite{campion-dense}, these universal properties apply in particular when $\calS = \square$ is the category of Gray cubes (cf. \cref{eg:cubes}). We regard this as a particularly canonical form of universal property for the Gray tensor product. But there are still some shortcomings (\cref{rmk:shortcomings}). So we also give a universal property (\cref{thm:gray-plain}) which is phrased in terms of the non-full subcategory $\square_\semi \subset \sCat_\omega$ of Gray cubes and subcomplex inclusions, i.e. the usual category of plain semicubical sets (\cref{def:semi}). This is a much smaller and more manageable category, still sufficient to detect the uniqueness of the Gray tensor product. Unfortunately, the uniqueness theorem in this case does not obviously descend along all of the localizations where the Gray tensor product is defined (\cref{eg:end}).

\begin{thm}\label{thm:gray-unique}
Let $L_S : \Cat_\omega^\inc {}^\to_\leftarrow L_S \Cat_\omega^\inc : i_S$ be a localization exhibiting $L_S \Cat_\omega^\inc$ as an exponential ideal in $\Cat_\omega^\inc$. Let $\TPar \subseteq \Cat_\omega^\inc$ be a full subcategory which is dense and closed under the Gray tensor product. Then the canonical functor $\TPar \to L_S \Cat_\omega^\inc$ is strong monoidal. Moreover, it enjoys the following two universal properties:
\begin{enumerate}
    \item\label{item:unique0.1} For any monoidal biclosed, cocomplete $\infty$-category $\calE$, composition with $\TPar \to L_S\Cat^\inc_\omega$ induces a fully faithful functor 
    \[\Fun_{E_1}^{\strong, \cocts}(L_S\Cat^\inc_\omega, \calE) \to \Fun_{E_1}^\strong(\TPar, \calE)\]
    whose essential image comprises those strong monoidal functors $\TPar \to \calE$ whose underlying functor extends to a cocontinuous functor $L_S \Cat_\omega^\inc \to \calE$.
    \item\label{item:unique0.2} The monoidal structure $\otimes_S$ on $L_S \Cat_\omega^\inc$ is the unique monoidal biclosed structure on $L_S \Cat_\omega^\inc$ such that $\TPar \to L_S \Cat_\omega^\inc$ is strong monoidal.
\end{enumerate}
\end{thm}
\begin{proof}
Just as in \cref{thm:gray-reflect}, this follows from \cite[Proposition 4.1.7.4]{ha}.
\end{proof}

\begin{rmk}
In \cref{thm:gray-unique}, we may take $\TPar$ to be a monoidally suitable site. In this case, we may sharpen the statement of \cref{thm:gray-unique}(\ref{item:unique0.1}) to the following. For any monoidal biclosed, cocomplete $\infty$-category $\calE$, the image of the fully faithful functor $\Fun_{E_1}^{\strong, \cocts}(L_S\Cat^\inc_\omega, \calE) \to \Fun_{E_1}^\strong(\TPar, \calE)$ comprises those strong monoidal functors $\TPar \to \calE$ which carry monic pushouts along folk cofibrations to pushouts.
\end{rmk}

\begin{eg}\label{eg:cubes}
In \cref{thm:gray-unique}, we may take $\TPar = \square$, the category of Gray cubes. Recall that $\square$ is not a monoidally suitable site because $\Theta \not \subseteq \square$. It is very \emph{nearly} suitable, though, as it is closed under the Gray tensor product by construction, and the main result of \cite{campion-dense} says that $\Theta$ is contained in the idempotent completion of $\square$. This is enough to deduce (as in \cite{campion-dense}) that $\square$ is dense in $\Cat_\omega^\inc$, so that \cref{thm:gray-unique} applies.
\end{eg}

\begin{rmk}\label{rmk:pre-shortcomings}
The $\infty$-category $\calS$ in \cref{thm:gray-unique} may be replaced with $L_S \calS$ (the image of $\calS$ under the functor $L_S : \Cat_\omega^\inc \to L_S \Cat_\omega^\inc$, viewed as a full subcategory of $L_S \Cat_\omega^\inc$). Then an analogous theorem holds: the functor $L_S \calS \to L_S \Cat_\omega^\inc$ is strong monoidal, with an analogous universal property to (1), and an analogous uniqueness property to (2). This statement is in some ways preferable to the given statement of \cref{thm:gray-unique} because the functor $L_S \calS \to L_S \Cat_\omega^\inc$ is fully faithful. However some caution is warranted because it may not be straightforward in general to compute the localization $L_S \calS$ (cf. \cref{rmk:shortcomings} below).
\end{rmk}

\begin{eg}\label{eg:favs}
Let $n \in \natsw$, let $\underline{d} : [n] \to \Levw$ a weakly decreasing multi-level, and let $R,C \in \natsw$. In \cref{thm:gray-unique}, we may take $L_S \Cat_\omega^\inc = \Cat_{\underline{d}}^{\inc < R}$ or $L_S \Cat_\omega^\inc = \Cat_{\underline{d}}^{\inc <^\coind C}$.
\end{eg}

Combining the two specializations (\cref{eg:cubes} and \cref{eg:favs}) of \cref{thm:gray-unique}, we obtain the following more specific uniqueness statements:

\begin{cor}\label{cor:gray-unique1}
Let $n \in \natsw$, let $\underline{d} : [n] \to \Levw$ be a weakly decreasing multi-level, and let $R,C \in \natsw$. Then the canonical functor $\square \to \Cat_{\underline{d}}^{\inc < R}$ is strong monoidal, and enjoys the following two universal properties:
\begin{enumerate}
    \item\label{item:unique1.1} For any monoidal biclosed, cocomplete $\infty$-category $\calE$, composition with $\square \to \Cat_{\underline{d}}^{\inc < R}$ induces a fully faithful functor 
    \[\Fun_{E_1}^{\strong, \cocts}(\Cat^{\inc < R}_{\underline{d}}, \calE) \to \Fun_{E_1}^\strong(\square, \calE)\]
    whose essential image comprises those strong monoidal functors $\square \to \calE$ whose underlying functor extends to a cocontinuous functor $\Cat_{\underline{d}}^{\inc < R} \to \calE$.
    \item\label{item:unique1.2} The Gray tensor product $\otimes$ on $\Cat_{\underline{d}}^{\inc < R}$ is the unique monoidal biclosed structure on $\Cat_{\underline{d}}^{\inc < R}$ such that $\square \to \Cat_{\underline{d}}^{\inc < R}$ is strong monoidal.
\end{enumerate}
Similar statements hold for $\square \to \Cat_{\underline{d}}^{\inc <^\coind C}$.
\end{cor}

\begin{eg}
The canonical functors $\square \to \Cat_n$, $\square \to \sCat_n$, $\square \to \Cat_{(m+n,n)}$ for $m,n \in \natsw$ are all strong monoidal with respect to the Gray tensor product, and the Gray tensor product is the unique monoidal biclosed structure making these functors strong monoidal.
\end{eg}

\begin{rmk}\label{rmk:shortcomings}
In the statement of \cref{cor:gray-unique1}, the $\infty$-category $\square$ in \cref{cor:gray-unique1} may be replaced with $L_{\underline{d}}^{\inc < R} \square$ or $L_{\underline{d}}^{\inc <^\coind C} \square$, to obtain a universal property for the Gray tensor product referring only to a small full monoidal subcategory of $\Cat_{\underline{d}}^{\inc < R}$ or $\Cat_{\underline{d}}^{\inc <^\coind C}$ (cf. \cref{rmk:pre-shortcomings}). However, care is warranted in computing the localization $L_{\underline{d}}^{\inc < R} \square$ or $L_{\underline{d}}^{\inc <^\coind C} \square$. For example, we do not know whether the objects of $L_2 \square$ are strict 2-categories. Consequently, we are unable to recover the main result of \cite{campion-maehara}, which says that there is a unique monoidal biclosed structure on $\Cat_2$ such that $L_2^s \square \to \Cat_2$ is strong monoidal, where $L_2^s \square$ is the reflection of $\square$ in $\sCat_2$. What we can deduce from \cref{cor:gray-unique1} is that there is a unique monoidal biclosed structure on $\Cat_2$ such that $L_2 \square \to \Cat_2$ is strong monoidal, and our knowledge of $L_2 \square$ is less explicit than our knowledge of $L_2^s \square$ (though it is quite likely that they do in fact coincide).
\end{rmk}

In the face of the shortcomings of \cref{cor:gray-unique1} identified in \cref{rmk:shortcomings}, we seek to provide some small comfort by characterizing the Gray tensor product using a smaller, non-full subcategory $\square_\semi \subset \square$.

\begin{Def}\label{def:semi}
Let $\square_\semi$ denote the plain semi-cube category. This is the free monoidal category on a bipointed object. Equivalently, it is the wide monoidal subcategory of $\square$ generated by the two face maps $i_0,i_1 : \square^0 \rightrightarrows \square^1$. Equivalently, $\square_\semi \subset \square$ is the category of Gray cubes and subcomputad inclusions between them.
\end{Def}


\begin{lem}\label{lem:po}
Let $\TPar$ be a suitable site and $\calC$ an $\infty$-category. Let $r : \square^m \to \globe_m$ be a split epimorphism. Then any functor $F : \TPar \to \calC$ which preserves monic pushouts along folk cofibrations also preserves the pushout $\globe_m = \square^m \ast_{\partial \square^m} \partial \globe_m$.
\end{lem}
\begin{proof}
This follows by cancellation with the monic pushout square $\square^m = \partial \square^m \ast_{\partial\globe_m} \globe_m$, since the composite is a pushout along the identity, preserved by any functor.
\end{proof}

\begin{prop}\label{prop:plain}
Let $\TPar$ be a monoidally suitable site. The Gray tensor product on $\TPar$ is the unique monoidal structure which preserves monic pushouts along folk cofibrations in each variable such that the inclusion $\square_\semi \to \TPar$ is strong monoidal.
\end{prop}
\begin{proof}
First, recall that every object of $\TPar$ may be built up via monic pushouts along $\partial \globe_m \to \globe_m$. By \cref{lem:po}, this means that our conditions determine $\otimes$ on objects, and in fact on the wide subcategory of subcomputad inclusions. This includes all of the coherence data for the monoidal structure. So it remains only to show that the bifunctor $\otimes$ is uniquely determined on more general morphisms. Factoring an arbitrary morphism $f \otimes g = (f \otimes \id)(\id \otimes g)$, we reduce to showing that the 1-variable functors $A \otimes (-)$ and $(-) \otimes A$ are uniquely determined for each $A \in \TPar$; we treat the former, as the latter is similar. Because $A$ is built up from cells and $\otimes$ is assumed to be compatible with these cell attachments, it suffices to treat the case where $A$ is a globe, or alternatively by \cref{lem:po}, the case where $A$ is a cube. By associativity of $\otimes$, it in fact suffices to treat the case where $A = \square^1$. Thus we are reduced to showing that the endofunctor $\square^1 \otimes (-)$ is uniquely determined on all morphisms. By density, it suffices to treat the case of a morphism $\zeta \to \theta$ with $\zeta,\theta \in \Theta$. Decomposing $\zeta$ as a colimit of cells, we reduce to the case where $\zeta = \globe_m$ is a globe. We may assume by induction that maps out of $W \otimes Z$ are determined when $W \otimes Z$ is lower-dimensional. In particular, maps out of $\partial \square^1 \otimes \globe_m$ and $\square^1 \otimes \partial \globe_m$ are determined, so that the map out of $\partial \square^1 \otimes \globe_m \cup_{\partial \square^1 \otimes \partial \globe_m} \square^1 \otimes \partial \globe_m$ is determined, and we just need to decide where to send the one remaining (top-dimensional) cell, knowing where its boundary is sent. Because $\Theta$ has Reedy factorizations, it suffices to consider the cases where $\globe_m \to \theta$ is injective or surjective.

If $\globe_m \to \theta$ is surjective, then we may assume that it is the degeneracy $\globe_m \to \globe_{m-1}$. In this case, the remaining cell of $\square^1 \otimes \globe_m$ is carried to a degenerate cell, so its image is uniquely determined by its boundary.

If $\globe_m \to \theta$ is injective, we may further factor through any outer face map (i.e. subcomputad inclusion) $\theta' \to \theta$ to assume that $\globe_m \to \theta$ is an inner face map. That is, $\globe_m \to \theta$ is the inclusion of the ``big cell" of $\theta$. We claim that in fact, the ``big cell" of $\square^1 \otimes \theta$ is the unique cell with its given boundary, so that must be where the big cell of $\square^1 \otimes \globe_m$ is carried. Otherwise, the big cell $\mu$ of $\theta$ decomposes nontrivially as $\mu = \mu_1 \circ_\alpha \mu_2$. So $\mu : \globe_m \to \theta$ factors as $\globe_m \to \globe_m \cup_{\globe_{|\alpha|}} \globe_k \xrightarrow{\mu_1 \cup \mu_2}\theta_1 \cup_{\globe_{|\alpha|}} \theta_2$. By induction, the maps $\square^1 \otimes \mu_i$ are determined, so the map $\mu_1 \cup \mu_2$ is as well by compatibility with monic pushouts along folk cofibrations. Thus we are reduced to the case where $\theta = \globe_m \cup_{|\alpha|} \globe_k$ (this includes as a degenerate case the case $\theta = \globe_k$). And indeed, in this case the boundary of the big cell of $\square^1 \otimes (\globe_m \cup_{|\alpha|} \globe_k)$ is $m$-dimensional and is not the boundary of an endomorphism, so any cell with this boundary must be $(m+1)$-dimensional. We can run through all of the $(m+1)$-dimensional cells explicitly and see that none of them have the same boundary as the ``big cell" except for the ``big cell" itself.
\end{proof}

\begin{thm}\label{thm:gray-plain}
Let $L_S \Cat_\omega^\inc$ be an exponential ideal in $\Cat_\omega^\inc$ such that $\TPar \to L_S \Cat_\omega^\inc$ is fully faithful for some monoidally suitable site $\TPar$. Then the Gray tensor product is the unique mononidal biclosed structure such that the inclusion $\square_\semi \to L_S \Cat_\omega^\inc$ is strong monoidal.
\end{thm}
\begin{proof}
Let $\TPar$ be such a monoidally suitable site. By \cref{cor:day-refl}, it suffices to show that the Gray tensor on $\TPar$ is uniquely determined as a monoidal structure by its restriction to $\square_\semi$, and compatibility with monic pushouts along folk cofibrations. So this follows from \cref{prop:plain}.
\end{proof}

\begin{eg}\label{eg:end}
The Gray tensor products on $\Cat_\omega$, $\Cat_\omega^\inc$, $\sCat_\omega$, $\Gaunt_\omega$, $\Cat_\omega^{\inc < R}$, $\Cat_\omega^{\inc <^\coind C}$ are all uniquely determined as monoidal biclosed structures by the stipulation that the inclusion of $\square_\semi$ be a strong monoidal functor. We do not know if the same holds when $L_S \Cat_\omega^\inc$ is $\Cat_n^\inc$, $\Cat_n$, etc. when $n$ is finite, because an analog of \cref{prop:plain} would require us to know whether $L_S(A \otimes B)$ is in $\TPar$ for $A,B \in \TPar$.
\end{eg}

\begin{rmk}
    We do not know whether the characterization in \cref{thm:gray-plain} can be extended to give a universal property for the Gray tensor product in terms of $\square_\semi$. We suspect that is this is not the case, but that it will become the case if we enlarge allow for some degeneracies in our cube category. For example, it may be that when $\calE$ is a cocomplete monoidal biclosed category, strong monoidality of cocontinuous functors $\Cat_\omega^\inc \to \calE$ is detected by restriction to the plain cube category $\square_\plain$. Such a universal property would mesh nicely with the fact that the plain cube category has a monoidal universal property: it is the free monoidal category on an augmented bipointed object. At issue is that the coherences for a strong monoidal structure on a functor out of $\square_\semi$ might not be natural with respect to all maps in $\square$: we are not sure how to resolve this even if $\calE$ is a 1-category, because although $\square_\plain$ generates $\square$ under colimits, it does not seem to be a dense generator.
\end{rmk}





\bibliographystyle{alpha}
\bibliography{gray}

\begin{thebibliography}{AABS02}

\bibitem[AABS02]{al-Agl;Brown;Steiner:Multiple}
Fahd~Ali Al-Agl, Ronald Brown, and Richard Steiner.
\newblock Multiple categories: The equivalence of a globular and a cubical
  approach.
\newblock {\em Advances in Mathematics}, 170(1):71--118, 2002.

\bibitem[AF18]{ayala-francis}
David Ayala and John Francis.
\newblock Flagged higher categories.
\newblock {\em Topology and quantum theory in interaction}, 718:137--173, 2018.

\bibitem[AL19]{ara-lucas}
Dimitri Ara and Maxime Lucas.
\newblock The folk model category structure on strict $\omega$-categories is
  monoidal.
\newblock {\em arXiv preprint arXiv:1909.13564}, 2019.

\bibitem[AM16]{ara-maltsiniotis}
Dimitri Ara and Georges Maltsiniotis.
\newblock Join and slices for strict $\infty$-categories.
\newblock {\em arXiv preprint arXiv:1607.00668}, 2016.

\bibitem[Ber07]{berger}
Clemens Berger.
\newblock Iterated wreath product of the simplex category and iterated loop
  spaces.
\newblock {\em Advances in Mathematics}, 213(1):230--270, 2007.

\bibitem[BSP21]{bsp}
Clark Barwick and Christopher Schommer-Pries.
\newblock On the unicity of the theory of higher categories.
\newblock {\em Journal of the American Mathematical Society}, 34(4):1011--1058,
  2021.

\bibitem[Cam22]{campion-dense}
Timothy Campion.
\newblock Cubes are dense in $(\infty,\infty)$-categories.
\newblock {\em arXiv preprint arXiv:2209.09376}, 2022.

\bibitem[Cam23]{campion-paste}
Timothy Campion.
\newblock An $(\infty,n)$-categorical pasting theorem.
\newblock {\em In preparation}, 2023.

\bibitem[CKM20]{Campion;Kapulkin;Maehara}
Tim Campion, Chris Kapulkin, and Yuki Maehara.
\newblock A cubical model for $(\infty, n)$-categories.
\newblock {\em arXiv preprint arXiv:2005.07603}, 2020.

\bibitem[CM23]{campion-maehara}
Timothy Campion and Yuki Maehara.
\newblock A model-independent gray tensor product for $(\infty, 2)$-categories.
\newblock {\em arXiv preprint arXiv:2304.05965}, 2023.

\bibitem[Cra95]{Crans:thesis}
Sjoerd~E Crans.
\newblock {\em On combinatorial models for higher-dimensional homotopies}.
\newblock PhD thesis, Utrecht University, 1995.

\bibitem[DKM23]{Doherty;Kapulkin;Maehara}
Brandon Doherty, Krzysztof Kapulkin, and Yuki Maehara.
\newblock Equivalence of cubical and simplicial approaches to
  $(\infty,n)$-categories.
\newblock {\em Advances in Mathematics}, 416:108902, 2023.

\bibitem[For19]{forest}
Simon Forest.
\newblock Unifying notions of pasting diagrams.
\newblock {\em arXiv preprint arXiv:1903.00282}, 2019.

\bibitem[For21]{forest-thesis}
Simon Forest.
\newblock {\em Computational descriptions of higher categories}.
\newblock PhD thesis, Institut Polytechnique de Paris, 2021.

\bibitem[GHL21]{Gagna;Harpaz;Lanari:Gray}
Andrea Gagna, Yonatan Harpaz, and Edoardo Lanari.
\newblock Gray tensor products and lax functors of $(\infty,2)$-categories.
\newblock {\em Advances in Mathematics}, 391:107986, 2021.

\bibitem[GHL22]{Gagna;Harpaz;Lanari:Equivalence}
Andrea Gagna, Yonatan Harpaz, and Edoardo Lanari.
\newblock On the equivalence of all models for $(\infty,2)$-categories.
\newblock {\em Journal of the London Mathematical Society}, 106(3):1920--1982,
  2022.

\bibitem[GR17]{Gaitsgory;Rozenblyum}
Dennis Gaitsgory and Nick Rozenblyum.
\newblock {\em A study in derived algebraic geometry. {V}ol. {I}.
  {C}orrespondences and duality}, volume 221 of {\em Mathematical Surveys and
  Monographs}.
\newblock American Mathematical Society, Providence, RI, 2017.

\bibitem[Gra74]{Gray}
John~W. Gray.
\newblock {\em Formal category theory: adjointness for {$2$}-categories}.
\newblock Lecture Notes in Mathematics, Vol. 391. Springer-Verlag, Berlin-New
  York, 1974.

\bibitem[Hen18]{henry-regular}
Simon Henry.
\newblock Regular polygraphs and the simpson conjecture.
\newblock {\em arXiv preprint arXiv:1807.02627}, 2018.

\bibitem[JFS17]{Johnson-Freyd;Scheimbauer}
Theo Johnson-Freyd and Claudia Scheimbauer.
\newblock (op) lax natural transformations, twisted quantum field theories, and
  “even higher” morita categories.
\newblock {\em Advances in Mathematics}, 307:147--223, 2017.

\bibitem[LMW10]{lmw}
Yves Lafont, Fran{\c{c}}ois M{\'e}tayer, and Krzysztof Worytkiewicz.
\newblock A folk model structure on omega-cat.
\newblock {\em Advances in Mathematics}, 224(3):1183--1231, 2010.

\bibitem[Lou22]{Loubaton}
F{\'e}lix Loubaton.
\newblock $n$-{C}omplicial sets as a model of $(\infty, n)$-categories.
\newblock {\em arXiv preprint arXiv:2207.08504}, 2022.

\bibitem[Lou23]{henry-loubaton}
Simon Henry~Felix Loubaton.
\newblock An inductive model structure for strict $\infty$-categories, 2023.

\bibitem[Lur]{ha}
Jacob Lurie.
\newblock Higher algebra.
\newblock available at \url{http://www.math.harvard.edu/~lurie/papers/HA.pdf}.

\bibitem[Lur09]{htt}
Jacob Lurie.
\newblock {\em Higher topos theory.}
\newblock Annals of Mathematics Studies, 170. Princeton University Press,
  Princeton, NJ, 2009.

\bibitem[Mae21]{Maehara:Gray}
Yuki Maehara.
\newblock The {G}ray tensor product for 2-quasi-categories.
\newblock {\em Adv. Math.}, 377:107461, 78, 2021.

\bibitem[M{\'e}t08]{metayer}
Fran{\c{c}}ois M{\'e}tayer.
\newblock Cofibrant objects among higher-dimensional categories.
\newblock {\em Homology, Homotopy and Applications}, 10(1):181--203, 2008.

\bibitem[Rez10]{rezk}
Charles Rezk.
\newblock A cartesian presentation of weak n--categories.
\newblock {\em Geometry \& Topology}, 14(1):521--571, 2010.

\bibitem[Ste04]{steiner}
Richard Steiner.
\newblock Omega-categories and chain complexes.
\newblock {\em Homology, Homotopy and Applications}, 6(1):175--200, 2004.

\bibitem[Str91]{street}
Ross Street.
\newblock Parity complexes.
\newblock {\em Cahiers de topologie et g{\'e}om{\'e}trie diff{\'e}rentielle
  cat{\'e}goriques}, 32(4):315--343, 1991.

\bibitem[Ver08]{Verity:I}
Dominic Verity.
\newblock Weak complicial sets. {I}. {B}asic homotopy theory.
\newblock {\em Adv. Math.}, 219(4):1081--1149, 2008.

\end{thebibliography}

\end{document}